\documentclass{amsart}
\usepackage{amssymb}
\usepackage{mathrsfs}
\usepackage{tikz-cd}
\usepackage[colorlinks = true]{hyperref}

\newcommand{\bk}{\Bbbk}
\newcommand{\Z}{\mathbb{Z}}

\newcommand{\F}{\mathbb{F}}

\renewcommand{\O}{\mathbb{O}}
\newcommand{\K}{\mathbb{K}}

\newcommand{\sA}{\mathscr{A}}

\newcommand{\bX}{\mathbf{X}}
\newcommand{\bXp}{\mathbf{X}^+}

\newcommand{\simto}{\overset{\sim}{\to}}
\newcommand{\id}{\mathrm{id}}
\newcommand{\op}{{\mathrm{op}}}
\newcommand{\Rep}{\mathrm{Rep}}

\DeclareMathOperator{\End}{End}
\DeclareMathOperator{\Hom}{Hom}
\DeclareMathOperator{\Ext}{Ext}
\DeclareMathOperator{\Irr}{Irr}

\DeclareMathOperator{\Ind}{Ind}

\newcommand{\ad}{\mathrm{ad}}
\newcommand{\uad}{\underline{\mathrm{ad}}}
\newcommand{\Spec}{\mathrm{Spec}}
\newcommand{\cO}{\mathcal{O}}

\makeatletter
\def\lotimes{\@ifnextchar_{\@lotimessub}{\@lotimesnosub}}
\def\@lotimessub_#1{\mathchoice{\mathbin{\mathop{\otimes}^L}_{#1}}%
  {\otimes^L_{#1}}{\otimes^L_{#1}}{\otimes^L_{#1}}}
\def\@lotimesnosub{\mathbin{\mathop{\otimes}^L}}
\makeatother

\numberwithin{equation}{section}
\newtheorem{thm}{Theorem}[section]
\newtheorem{lem}[thm]{Lemma}
\newtheorem{prop}[thm]{Proposition}

\theoremstyle{definition}

\theoremstyle{remark}
\newtheorem{rmk}[thm]{Remark}

\title[Representation theory of disconnected reductive groups]{Representation theory of\\ disconnected reductive groups}

 \author{Pramod N. Achar}
 \address{Department of Mathematics\\
   Louisiana State University\\
   Baton Rouge, LA 70803\\
   U.S.A.}
 \email{pramod@math.lsu.edu}
 
  \author{William Hardesty}
   \address{Department of Mathematics\\
   Louisiana State University\\
   Baton Rouge, LA 70803\\
   U.S.A.}
  \email{whardesty@lsu.edu}
 
 \author{Simon Riche}
 \address{Universit\'e Clermont Auvergne, CNRS, LMBP, F-63000 Clermont-Ferrand, France.
 }
 \email{simon.riche@uca.fr}
 
\thanks{P.A. was supported by NSF Grant Nos.~DMS-1500890 and DMS-1802241. This project has received funding from
the European Research Council (ERC) under the European Union's Horizon 2020 research and
innovation programme (grant agreement No.~677147).}
 
\begin{document}
 
\begin{abstract}
We study three fundamental topics in the representation theory of disconnected algebraic groups whose identity component is reductive: (i)~the classification of irreducible representations; (ii)~the existence and properties of Weyl and dual Weyl modules; and  (iii)~the decomposition map relating representations in characteristic $0$ and those in characteristic $p$ (for groups defined over discrete valuation rings of mixed characteristic). For each of these topics, we obtain natural generalizations of the well-known results for connected reductive groups.
 \end{abstract}
 
\maketitle

%%%%%%%%%%%%%%%%%%%%%%%%%%%%%%%%%%%%%%%%%%%%%%%%%%%%%%%%%%%%%%%%%%%%%%%%%%%
\section{Introduction}
\label{sec:intro}
%%%%%%%%%%%%%%%%%%%%%%%%%%%%%%%%%%%%%%%%%%%%%%%%%%%%%%%%%%%%%%%%%%%%%%%%%%%

Let $G$ be a (possibly disconnected) affine algebraic group over an algebraically closed field $\bk$, and let $G^\circ$ be its identity component.  We call $G$ a \emph{(possibly) disconnected reductive group} if $G^\circ$ is reductive.  The goal of this paper is to extend a number of well-known foundational facts about connected reductive groups to the disconnected case.  

Such groups occur naturally, even when one is primarily interested in \emph{connected} reductive groups. Namely, for a connected reductive group $H$, the stabilizer $H^x$ of a nilpotent element in the Lie algebra of $H$ may be disconnected.  Let $H^x_{\mathrm{unip}}$ be its unipotent radical; then $H^x/H^x_{\mathrm{unip}}$ is a disconnected reductive group.  The study of (the derived category of) coherent sheaves on the nilpotent cone $\mathcal{N}$ of $H$, and in particular of \emph{perverse-coherent sheaves} on $\mathcal{N}$, leads naturally to questions about representations of $H^x/H^x_{\mathrm{unip}}$.  See~\cite{ahr} for some questions of this form, and for some applications of the results of this paper.

The present paper contains three main results:
\begin{enumerate}
\item
We classify the irreducible representations of $G$ in terms of those of $G^\circ$, via an adaptation of Clifford theory (Theorem~\ref{thm:disconn-class}).
\item
Assuming that the characteristic of $\bk$ does not divide $|G/G^\circ|$, we prove that the category of finite-dimensional $G$-modules has a natural structure of a highest-weight category (Theorem~\ref{thm:hw-structure}).  
\item
Starting from a disconnected reductive group scheme over a strictly Hen\-sel\-ian discrete valuation ring of mixed characteristic,
one obtains a ``decomposition map'' relating the Grothendieck groups of representations in characteristic $0$ and in characteristic $p$.  We prove that this map is an isomorphism.
\end{enumerate}
These results are certainly not surprising, and
some of them may be known to experts, but we are not aware of a reference that treats them in the detail and generality needed for the applications in~\cite{ahr}.

\subsection*{Acknowledgments}
We thank Jens Carsten Jantzen for a helpful conversation, and George Lusztig for sharing with us some unpublished notes on a topic related to the present paper.  We are also grateful to Torsten Wedhorn for explaining to us how to both simplify and extend the generality of Section~\ref{sec:Groth}.

%%%%%%%%%%%%%%%%%%%%%%%%%%%%%%%%%%%%%%%%%%%%%%%%%%%%%%%%%%%%%%%%%%%%%%%%%%%
\section{Classification of simple representations}
\label{sec:simples}
%%%%%%%%%%%%%%%%%%%%%%%%%%%%%%%%%%%%%%%%%%%%%%%%%%%%%%%%%%%%%%%%%%%%%%%%%%%

In this section we consider (affine) algebraic groups over an arbitrary
algebraically closed field $\bk$. Our goal is to describe the representation theory of a disconnected algebraic group $G$ whose neutral connected component $G^\circ$ is reductive in terms of the representation theory of $G^\circ$, via a kind of Clifford theory.

%--------------------------------------------------------------------------
\subsection{Twist of a representation by an automorphism}
\label{app:twist}
%--------------------------------------------------------------------------

Let $G$ be an algebraic group, $\varphi : G \simto G$ an automorphism, and let $\pi=(V,\varrho)$ be a representation of $G$. Then we define the representation ${}^\varphi \hspace{-1pt} \pi$ as the pair $(V,\varrho \circ \varphi^{-1})$. (Below, we will most of the time write $V$ for $\pi$, and ${}^\varphi \hspace{-1pt} V$ for ${}^\varphi \hspace{-1pt} \pi$.) It is straightforward to check that if $\psi : G \simto G$ is a second automorphism, then we have
\begin{equation}
\label{eqn:twist-composition}
{}^\psi \hspace{-1pt} \bigl( {}^\varphi \hspace{-1pt} \pi \bigr) = {}^{\psi \circ \varphi} \hspace{-1pt} \pi.
\end{equation}
If $f : \pi \to \pi'$ is a morphism of $G$-representations, then the same linear map defines a morphism of $G$-representations ${}^\varphi \hspace{-1pt} \pi \to {}^\varphi \hspace{-1pt} \pi'$, which will sometimes be denoted ${}^\varphi \hspace{-1pt} f$.

\begin{lem}
\label{lem:twist-Ind}
Let $H \subset G$ be a subgroup, and $(V,\varrho)$ be a representation of $H$. Then there exists a canonical isomorphism of $G$-modules
\[
{}^\varphi \hspace{-1pt} \Ind_H^G(V,\varrho) \cong \Ind_{\varphi(H)}^{G}(V, \varrho \circ \varphi^{-1}).
\]
\end{lem}

\begin{proof}
By definition, we have
\begin{align*}
\Ind_H^G(V,\varrho) &= \{f : G \to V \mid \forall h \in H, \, f(gh)=\varrho(h^{-1})(f(g))\}, \\
\Ind_{\varphi(H)}^G(V,\varrho \circ \varphi^{-1}) &= \{f : G \to V \mid \forall h \in \varphi(H), \, f(gh)=\varrho \circ \varphi^{-1}(h^{-1})(f(g))\}.
\end{align*}
Here, in both cases the functions are assumed to be algebraic, and the $G$-action is defined by $(g \cdot f)(h)=f(g^{-1}h)$. We have a natural isomorphism of vector spaces
\[
\Ind_H^G(V,\varrho) \simto \Ind_{\varphi(H)}^G(V,\varrho \circ \varphi^{-1})
\]
sending $f$ to $f \circ \varphi^{-1}$. It is straightforward to check that this morphism is an isomorphism of $G$-modules from ${}^\varphi \hspace{-1pt} \Ind_H^G(V,\varrho)$ to $\Ind_{\varphi(H)}^{G}(V, \varrho \circ \varphi^{-1})$.
\end{proof}

\begin{rmk}
\label{rmk:twist-Ind}
More generally, if $G'$ is another algebraic group and $\varphi : G \simto G'$ is an isomorphism, for any $G$-module $\pi$ we can consider the $G'$ module ${}^\varphi \hspace{-1pt} \pi$ defined as above. Then the same arguments as for Lemma~\ref{lem:twist-Ind} show that we have ${}^\varphi \hspace{-1pt} \Ind_H^G(\pi) \cong \Ind_{\varphi(H)}^{G'}({}^\varphi \hspace{-1pt} \pi)$.
\end{rmk}

In particular, assume that we are given an algebraic group $G'$ and an embedding of $G$ as a normal subgroup of $G'$. Then for any $g \in G'$, we have an automorphism $\ad(g)$ of $G$ sending $h$ to $ghg^{-1}$. In this setting, we will write ${}^g \hspace{-1pt} V$ for ${}^{\ad(g)} \hspace{-1pt} V$, and ${}^g \hspace{-1pt} f$ for ${}^{\ad(g)} \hspace{-1pt} f$. Then for $g,h \in G'$, since $\ad(g) \circ \ad(h) = \ad(gh)$, \eqref{eqn:twist-composition} translates to ${}^g \hspace{-1pt} \bigl( {}^h \hspace{-1pt} V \bigr) = {}^{gh} \hspace{-1pt} V$.

The verification of the following lemma is straightforward.

\begin{lem}
\label{lem:twist-inner}
Let $(V,\varrho)$ be a representation of $G$. Then if $g \in G$, $\varrho(g^{-1})$ induces an isomorphism $V \simto {}^g \hspace{-1pt} V$. 
\end{lem}

%--------------------------------------------------------------------------
\subsection{Disconnected reductive groups}
\label{ss:disconn}
%--------------------------------------------------------------------------

From now on we fix
an algebraic group $G$ whose identity component $G^\circ$ is reductive. 
We set $A := G/G^\circ$ (a finite group).  
The canonical quotient morphism $G \to A$ will be denoted $\varpi$.

Let $T$ be the ``universal maximal torus'' of $G^\circ$, i.e., the quotient $B/(B,B)$ for any Borel subgroup $B \subset G^\circ$. (Since all Borel subgroups in $G^\circ$ are $G^\circ$-conjugate, and since $B=N_{G^\circ}(B)$ acts trivially on $B/(B,B)$, the quotient $B/(B,B)$ 
does not depend on $B$, up to canonical isomorphism.)  Let $\bX = X^*(T)$ be its weight lattice.
If $T' \subset B$ is any maximal torus, then the composition $T' \hookrightarrow B \twoheadrightarrow T$ is an isomorphism, and this lets us identify $\bX$ with $X^*(T')$.  The image in $\bX$ under this identification of the roots of $(G,T')$, and of the subset of positive roots (chosen as the opposite of the $T'$-weights on the Lie algebra of $B$), do not depend on the choice of $T'$; so they define the canonical root system $\Phi \subset \bX$ and the subset $\Phi^+ \subset \Phi$ of positive roots. Similar comments apply to coroots, so that we can define the dominant weights $\bXp \subset \bX$. We denote by $W$ the Weyl group of $T$. (This group is well defined because $N_{B}(T')=T'$ for a maximal torus $T'$ contained in a Borel subgroup $B$.)
 
Given a weight $\lambda \in \bXp$, we denote by
\[
L(\lambda),\qquad \Delta(\lambda),\qquad \nabla(\lambda)
\]
the irreducible, Weyl, and dual Weyl $G^\circ$-modules, respectively, corresponding to $\lambda$. Here $\nabla(\lambda)$ is defined as the induced module $\Ind_{B}^{G^\circ}(\bk_B(\lambda))$ for some choice of Borel subgroup $B \subset G^\circ$, $L(\lambda)$ is the unique simple submodule of $\nabla(\lambda)$, and $\Delta(\lambda)$ is defined as $(\nabla(-w_0 \lambda))^*$, where $w_0 \in W$ is the longest element. (These modules do not depend on the choice of $B$ up to isomorphism thanks to Lemma~\ref{lem:twist-Ind} and Lemma~\ref{lem:twist-inner}.)

For any $g \in G$ and any Borel subgroup $B \subset G^\circ$, $\ad(g)$ induces an isomorphism $B/(B,B) \simto gBg^{-1}/(gBg^{-1}, gBg^{-1})$. Since $gBg^{-1}$ is also a Borel subgroup of $G^\circ$, this defines an automorphism $\uad(g)$ of $T$. Explicitly, we can choose an $h \in G^\circ$ such that $gBg^{-1} = hBh^{-1}$, and then for any element $b(B,B) \in B/(B,B) = T$, we set
\[
\uad(g)(b(B,B)) = h^{-1}gbg^{-1}h(B,B).
\]
It is straightforward to check that the right-hand side is independent of $h$. The fact that $T$ is well defined translates to the property that $\uad(g)=\id$ if $g \in G^\circ$, so that $\uad$ factors through a morphism $A \to \mathrm{Aut}(T)$, which we will also denote by $\uad$.

For $a \in A$ and $\lambda \in \bX$, we set 
\begin{equation}\label{eqn:combinatorial-action}
{}^a \hspace{-1pt} \lambda := \lambda \circ \uad(a^{-1}).
\end{equation}
This operation defines an action of $A$ on $\bX$.  Now let $g \in \varpi^{-1}(a) \subset G$, and let $T' \subset B$ be a maximal torus.  There is an $h \in G^\circ$ such that $gT'g^{-1} = hT'h^{-1}$ and $gBg^{-1} = hBh^{-1}$. Then $h^{-1}g$ normalizes $B$ and $(B,B)$.  If $x$ is a root vector for $T'$ in the Lie algebra of $(B,B)$, say with root $\lambda \in -\Phi^+$, then $\mathrm{Ad}(h^{-1}g)(x)$ is also a root vector with root ${}^a\hspace{-1pt}\lambda$.  This shows that the action of $A$ on $\bX$ preserves $\Phi^+$ and $\Phi$.  Similar reasoning shows that it preserves $\bXp$.
Moreover, Lemma~\ref{lem:twist-Ind} implies that for any $\lambda \in \bXp$ and $g \in G$, we have canonical isomorphisms
\begin{equation}
\label{eqn:twist-standard}
{}^g \hspace{-1pt} \Delta(\lambda) \cong \Delta({}^{\varpi(g)} \hspace{-1pt} \lambda), \quad
{}^g \hspace{-1pt} L(\lambda) \cong L({}^{\varpi(g)} \hspace{-1pt} \lambda), \quad
{}^g \hspace{-1pt} \nabla(\lambda) \cong \nabla({}^{\varpi(g)} \hspace{-1pt} \lambda).
\end{equation}

We will denote by $\Irr(G^\circ)$ the set of isomorphism classes of simple $G^\circ$-modules. This set admits an action of $G$, where $g$ acts via $[V] \mapsto [{}^g \hspace{-1pt} V]$. (Of course, this action factors through an action of $A$.) The constructions above provide a natural bijection $\bX^+ \simto \Irr(G^\circ)$ (sending $\lambda$ to the isomorphism class of $L(\lambda)$), which is $A$-equivariant in view of~\eqref{eqn:twist-standard}.

\begin{lem}
\label{lem:restrict-semis}
Let $V$ be an irreducible $G$-module.  Then $V$ is semisimple as a $G^\circ$-module.  All of its irreducible $G^\circ$-submodules lie in a single $G$-orbit in $\Irr(G^\circ)$.
\end{lem}

\begin{proof}
Choose an irreducible $G^\circ$-submodule $M \subset V$, and choose a set of coset representatives $g_1, \ldots, g_r$ for $G^\circ$ in $G$.  The subspace
\[
\sum_{i=1}^r g_iM \subset V
\]
is stable under the action of $G$, so it must be all of $V$.  Each summand $g_iM$ is stable under $G^\circ$, so there is a surjective map of $G^\circ$-representations
\[
\bigoplus_{i=1}^r g_i M \to \sum_{i=1}^r g_iM = V.
\]
Now, $g_iM$ is isomorphic as a $G^\circ$-module to ${}^{g_i} \hspace{-1pt} M$; in particular, each $g_iM$ is an irreducible $G^\circ$-module, and $\bigoplus_i g_iM$ is semisimple.  Thus, as a $G^\circ$-module, $V$ is a quotient of a semisimple module, all of whose summands lie in a single $G$-orbit of $\Irr(G^\circ)$, so the same holds for $V$ itself.
\end{proof}

%--------------------------------------------------------------------------
\subsection{The component group and induced representations}
\label{ss:component}
%--------------------------------------------------------------------------

For each $a \in A = G/G^\circ$, let us choose, once and for all, a representative $\iota(a) \in G$.  In the special case $a = 1_A$, we require that
\[
\iota(1_A) = 1_G.
\]
Given $a,b \in A$, the representative $\iota(ab)$ need not be equal to $\iota(a)\iota(b)$; but these elements lie in the same coset of $G^\circ$.  Explicitly, there is a unique element $\gamma(a,b) \in G^\circ$ such that
\[
\iota(a)\iota(b) = \iota(ab)\gamma(a,b).
\]
Our assumption on $\iota(1_A)$ implies that for any $a \in A$, we have
\[
\gamma(1_A,a) = \gamma(a,1_A) = 1_G.
\]
By expanding $\iota(abc)$ in two ways, one finds that
\begin{equation}
\label{eqn:gamma-product}
\gamma(ab,c) \cdot \ad(\iota(c)^{-1})(\gamma(a,b)) = \gamma(a,bc)\gamma(b,c).
\end{equation}

Now let $V$ be a $G^\circ$-module.  By Lemma~\ref{lem:twist-inner}, for any $a,b \in A$ the action of $\gamma(a,b)$ defines an isomorphism of $G^\circ$-modules
\[
{}^{\gamma(a,b)} \hspace{-1pt} V \simto V.
\]
Twisting by $\iota(ab)$ we deduce an isomorphism
\[
\phi_{a,b} : {}^{\iota(a)\iota(b)} \hspace{-1pt} V \simto {}^{\iota(ab)} \hspace{-1pt} V.
\]

We can use the maps $\iota$ and $\gamma$ to explicitly describe representations of $G$ that are induced from $G^\circ$, as follows. Let us denote by $\bk[A]$ the group algebra of $A$ over $\bk$. Let $V$ be a $G^\circ$-module, and consider the vector space
\begin{equation}\label{eqn:induced-explicit}
\tilde V = \bk[A] \otimes V = \bigoplus_{f \in A} \bk f \otimes V.
\end{equation}
We now explain how to make $\tilde V$ into a $G$-module.  Note that every element of $G$ can be written uniquely as $\iota(a)g$ for some $a \in A$ and $g \in G^\circ$.  We put
\begin{equation}\label{eqn:induced-action}
\iota(a)g \cdot (f \otimes v) = af \otimes \gamma(a,f) \cdot \ad(\iota(f)^{-1})(g) \cdot v.
\end{equation}
Using~\eqref{eqn:gamma-product} one can
check that this does indeed define an action of $G$ on $\tilde V$.  

\begin{lem}
\label{lem:ind-tensor}
The map
\[
f \mapsto \sum_{a \in A} a \otimes f(\iota(a))
\]
defines an isomorphism of $G$-modules $\Ind_{G^\circ}^G(V) \simto \tilde{V}$.
\end{lem}

\begin{proof}
It is clear that our map is an isomorphism of vector spaces, and that its inverse sends $a \otimes v$ to the function $f : G \to V$ such that $f(\iota(a)g) = g^{-1} \cdot v$ for $g \in G^\circ$ and $f(\iota(b)g) = 0$ for $g \in G^\circ$ and $b \in A \smallsetminus \{a\}$. It is not difficult to check that this inverse map respects the $G$-actions, proving the proposition.
\end{proof}

In view of Lemma~\ref{lem:ind-tensor}, it is clear
that as $G^\circ$-modules, we have
\begin{equation}\label{eqn:induce-restrict}
\Ind_{G^\circ}^G(V) \cong \bigoplus_{f \in A} {}^{\iota(f)} \hspace{-1pt} V,
\end{equation}
as expected.

%--------------------------------------------------------------------------
\subsection{A twisted group algebra of a stabilizer}
\label{ss:twisted}
%--------------------------------------------------------------------------

Let $\lambda \in \bXp$, and let $A^\lambda=\{a \in A \mid {}^a \hspace{-1pt} \lambda = \lambda\}$ be its stabilizer. We also set $G^\lambda := \varpi^{-1}(A^\lambda)$. In view of~\eqref{eqn:twist-standard}, we have
\begin{equation}\label{eqn:stab-subgroup}
G^\lambda = \{g \in G \mid {}^g \hspace{-1pt} L(\lambda) \cong L(\lambda)\}.
\end{equation}

We fix a representative for the simple $G^\circ$-module $L(\lambda)$ and,
for each $a \in A^\lambda$, an isomorphism of $G^\circ$-modules
\[
\theta_a: L(\lambda) \simto {}^{\iota(a)} \hspace{-1pt} L(\lambda).
\]
In the special case that $a = 1_A$, we require that
\[
\theta_{1_A} = \id_{L(\lambda)}.
\]
Explicitly, these maps have the property that for any $g \in G^\circ$ and $v \in L(\lambda)$, we have
\begin{equation}
\label{eqn:def-thetaa}
\theta_a(g \cdot v) = \ad(\iota(a)^{-1})(g) \cdot \theta_a(v),
\end{equation}
where on the right-hand side we consider the given action of $G^\circ$ on $L(\lambda)$.

Now let $a,b \in A^\lambda$, and consider the diagram
\[
\begin{tikzcd}
L(\lambda) \ar[r, "\theta_a"] \ar[rrr, bend left=30, "\theta_{ab}"] &
{}^{\iota(a)} \hspace{-1pt} L(\lambda) \ar[r, "{}^{\iota(a)} \hspace{-1pt} \theta_b"] &
{}^{\iota(a)\iota(b)} \hspace{-1pt} L(\lambda) \ar[r, "{\phi_{a,b}}"] &
{}^{\iota(ab)} \hspace{-1pt} L(\lambda).
\end{tikzcd}
\]
This is \emph{not} a commutative diagram.  Rather, both $\theta_{ab}$ and $\phi_{a,b} \circ {}^{\iota(a)}\theta_b \circ \theta_a$ are isomorphisms of simple $G^\circ$-modules, so they must be scalar multiples of one another.  Let $\alpha(a,b) \in \bk^\times$ be the scalar such that
\[
\phi_{a,b} {}^{\iota(a)}\theta_b \theta_a = \alpha(a,b) \cdot \theta_{ab}.
\]
Our assumptions on $\iota(1_A)$ and $\theta_{1_A}$ imply that for all $a \in A$, we have
\[
\alpha(1_A,a) = \alpha(a,1_A) = 1.
\]

\begin{figure}
\[
\begin{tikzcd}[column sep=11pt]
&&& {}^{\iota(ab)} \hspace{-1pt} L(\lambda) \ar[r, "\theta_c"] &
 {}^{\iota(ab)\iota(c)} \hspace{-1pt} L(\lambda) \ar[dr, "{\phi_{ab,c}}" description] \\
L(\lambda) \ar[r, "\theta_a"] \ar[urrr, bend left=20, "\theta_{ab}"] \ar[rrrrr, bend left=50, "\theta_{abc}"] &
{}^{\iota(a)} \hspace{-1pt} L(\lambda) \ar[r, "\theta_b"] \ar[drrr, bend right=40, "\theta_{bc}"] &
{}^{\iota(a)\iota(b)} \hspace{-1pt} L(\lambda) \ar[ur, "{\phi_{a,b}}"] \ar[dr, "\theta_c"] &&&
{}^{\iota(abc)} \hspace{-1pt} L(\lambda) \\
&&& {}^{\iota(a)\iota(b)\iota(c)} \hspace{-1pt} L(\lambda) \ar[r, "{\phi_{b,c}}"] \ar[uur, "{\ad(\iota(c)^{-1})(\gamma(a,b)) \cdot (-)}" description] &
  {}^{\iota(a)\iota(bc)} \hspace{-1pt} L(\lambda) \ar[ur, "{\phi_{a,bc}}"]
\end{tikzcd}
\]
\caption{Isomorphisms of $L(\lambda)$ with ${}^{\iota(abc)} \hspace{-1pt} L(\lambda)$}\label{fig:cocycle}
\end{figure}
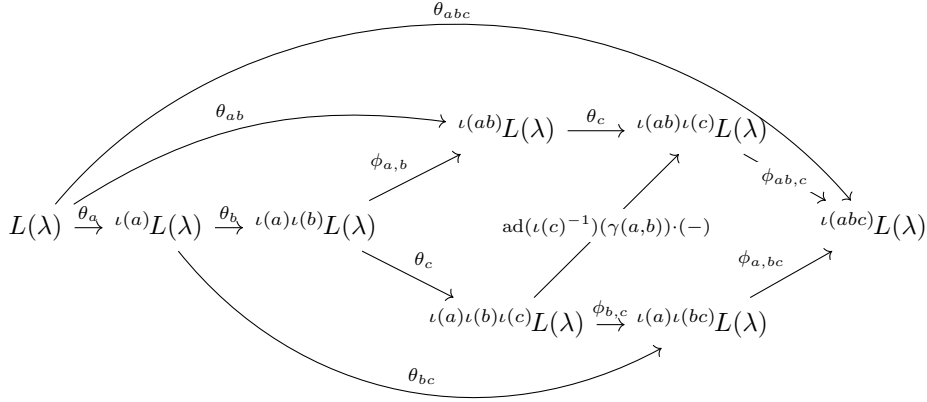

Given three elements $a,b,c \in A^\lambda$, we can form the diagram shown in Figure~\ref{fig:cocycle}.  The subdiagram consisting of straight arrows is commutative (by~\eqref{eqn:gamma-product},~\eqref{eqn:def-thetaa} and the definitions), whereas each curved arrow introduces a scalar factor.  Comparing the different scalars shows that
\[
\alpha(a,b)\alpha(ab,c) = \alpha(a,bc) \alpha(b,c).
\]
In other words, $\alpha: A^\lambda \times A^\lambda \to \bk^\times$ is a $2$-cocycle.

Let $\sA^\lambda$ be the twisted group algebra of $A^\lambda$ determined by this cocycle.  Explicitly, we define $\sA^\lambda$ to be the $\bk$-vector space spanned by symbols $\{ \rho_a : a \in A^\lambda \}$ with multiplication given by
\[
\rho_a \rho_b = \alpha(a,b) \rho_{ab}.
\]
This is a unital $\bk$-algebra, with unit $\rho_{1_A}$. 

The algebra $\sA^\lambda$ can be described in more canonical terms as follows.

\begin{prop}
\label{prop:end-induced}
There exists a canonical isomorphism of $\bk$-algebras
\[
\End_{G^\lambda} \bigl( \Ind_{G^\circ}^{G^\lambda} (L(\lambda)) \bigr) \cong (\sA^\lambda)^\op.
\]
\end{prop}

\begin{proof}
We will work with the description of $\Ind_{G^\circ}^{G^\lambda} (L(\lambda))$ from Lemma~\ref{lem:ind-tensor} (applied to the group $G^\lambda$):
we identify it with $\bk[A^\lambda] \otimes L(\lambda)$, where the action of $G^\lambda$ is given by~\eqref{eqn:induced-action}.

We begin by equipping $\bk[A^\lambda] \otimes L(\lambda)$ with the structure of a right $\sA^\lambda$-module as follows: given $a,f \in A^\lambda$ and $v \in L(\lambda)$, we put
\begin{equation}\label{eqn:sA-action}
(f \otimes v) \cdot \rho_a := (fa) \otimes \gamma(f,a) \cdot \theta_a(v).
\end{equation}
Let us check that this is indeed a right $\sA^\lambda$-module structure:
\begin{align*}
((f \otimes v) \cdot \rho_a) \cdot \rho_b
&= ((fa) \otimes \gamma(f,a) \cdot \theta_a(v)) \cdot \rho_b \\
&= (fab) \otimes \gamma(fa,b) \cdot \theta_b(\gamma(f,a) \cdot \theta_a(v)) \\
&= (fab) \otimes \gamma(fa,b) \ad(\iota(b)^{-1})(\gamma(f,a)) \cdot \theta_b(\theta_a(v)) \\
&= (fab) \otimes \gamma(f,ab) \gamma(a,b) \cdot \theta_b(\theta_a(v)) \\
&= (fab) \otimes \alpha(a,b)(\gamma(f,ab) \cdot \theta_{ab}(v)) \\
&= (f \otimes v) \cdot (\alpha(a,b)\rho_{ab}).
\end{align*}
(Here, the third equality relies on~\eqref{eqn:def-thetaa}, and the fourth one on~\eqref{eqn:gamma-product}.)
Next, we check that the right action of $\sA^\lambda$ commutes with the left action of $G$:
\begin{align*}
\iota(a)g \cdot ((f \otimes v) \cdot \rho_b) 
&= \iota(a)g \cdot ((fb) \otimes \gamma(f,b) \cdot \theta_b(v)) \\
&= (afb) \otimes \gamma(a,fb)\ad(\iota(fb)^{-1})(g)\gamma(f,b) \cdot \theta_b(v) \\
&= (afb) \otimes \gamma(a,fb) \gamma(f,b) \ad((\iota(fb)\gamma(f,b))^{-1})(g) \cdot \theta_b(v) \\
&= (afb) \otimes \gamma(af,b) \ad(\iota(b)^{-1})(\gamma(a,f)) \ad((\iota(f)\iota(b))^{-1})(g) \cdot \theta_b(v) \\
&= (afb) \otimes \gamma(af,b) \theta_b(\gamma(a,f) \ad(\iota(f)^{-1})(g) \cdot v) \\
&= ((af) \otimes \gamma(a,f) \ad(\iota(f)^{-1})(g) \cdot v) \cdot \rho_b \\
&= (\iota(a)g \cdot (f \otimes v)) \cdot \rho_b.
\end{align*}

As a consequence, the right $\sA^\lambda$-action gives rise to an algebra homomorphism
\[
\varphi: (\sA^\lambda)^\op \to \End_{G^\lambda}(\bk[A^\lambda] \otimes L(\lambda)).
\]
For each $a \in A^\lambda$, the operator $\varphi(\rho_a)$ permutes the direct summands $\bk f \otimes L(\lambda) \subset \bk[A^\lambda] \otimes L(\lambda)$, as $f$ runs over elements of $A^\lambda$.  Moreover, distinct $a$'s give rise to distinct permutations.  It follows from this that the collection of linear operators $\{ \varphi(\rho_a) : a \in A^\lambda \}$ is linearly independent.  In other words, $\varphi$ is injective.

On the other hand, by adjunction, we have
\begin{equation}
\label{eqn:dim-End}
\dim \End_{G^\lambda} \bigl( \Ind_{G^\circ}^{G^\lambda} (L(\lambda)) \bigr)
= \dim \Hom_{G^\circ} \bigl( \Ind_{G^\circ}^{G^\lambda} (L(\lambda)), L(\lambda) \bigr).
\end{equation}
Now,~\eqref{eqn:induce-restrict} implies that as a $G^\circ$-module, $\Ind_{G^\circ}^{G^\lambda} (L(\lambda))$ is isomorphic to a direct sum of $|A^\lambda|$ copies of $L(\lambda)$.  So~\eqref{eqn:dim-End} shows that
\[
\dim \End_{G^\lambda} \bigl( \Ind_{G^\circ}^{G^\lambda} (L(\lambda)) \bigr) = |A^\lambda| = \dim \sA^\lambda.
\]
Since $\varphi$ is an injective map between $\bk$-vector spaces of the same dimension, it is also surjective, and hence an isomorphism.
\end{proof}

\begin{rmk}\phantomsection
\label{rmk:indep-1}
\begin{enumerate}
\item
The $G^\circ$-module $L(\lambda)$ is defined only up to isomorphism. But if $L'(\lambda)$ is another choice for this module, then an isomorphism $L(\lambda) \simto L'(\lambda)$ is unique up to scalar (and exists). Hence the induced isomorphism $\End_{G^\lambda} \bigl( \Ind_{G^\circ}^{G^\lambda} (L(\lambda)) \bigr) \simto \End_{G^\lambda} \bigl( \Ind_{G^\circ}^{G^\lambda} (L'(\lambda)) \bigr)$ does not depend on the choice of isomorphism. In other words, the algebra $\End_{G^\lambda} \bigl( \Ind_{G^\circ}^{G^\lambda} (L(\lambda)) \bigr)$ is completely canonical, i.e.~does not depend on any choice.
\item
Once the $G^\circ$-module $L(\lambda)$ is fixed,
our description of the $\bk$-algebra $\sA^\lambda$, and of its identification with $\End_{G^\lambda}(\Ind_{G^\circ}^{G^\lambda}(L(\lambda)))^\op$ in Proposition~\ref{prop:end-induced}, depend on the choice of the isomorphisms $\theta_a$ for $a \in A \smallsetminus \{1\}$. However, if $\{\theta'_a : a \in A \smallsetminus \{1\}\}$ is another choice of such isomorphisms, and $\{\rho'_a : a \in A\}$ is the basis of the corresponding algebra $(\sA')^\lambda$, then for any $a \in A$ there exists a unique $t_a \in \bk^\times$ such that $\theta'_a = t_a \theta_a$. It is easy to check that the assignment $\rho'_a \mapsto t_a \rho_a$ defines an algebra isomorphism $(\sA')^\lambda \simto \sA^\lambda$ which commutes with the identifications provided by Proposition~\ref{prop:end-induced}.
\item
\label{it:action-A-Ind}
If, instead of using Lemma~\ref{lem:ind-tensor} to describe the $G^\lambda$-module $\Ind_{G^\circ}^{G^\lambda} (L(\lambda))$, we describe it in terms of algebraic functions $\phi : G^\lambda \to L(\lambda)$ satisfying $\phi(gh)=h^{-1} \cdot \phi(g)$ for $h \in G^\circ$, then the right action of $\sA^\lambda$ on this module satisfies $(\phi \cdot \rho_a)(g) = \theta_a \circ \phi(g\iota(a)^{-1})$.
\end{enumerate}
\end{rmk}

%-----------------------------------------------------------------------
\subsection{Simple \texorpdfstring{$G^\lambda$}{Glambda}-modules whose restriction to \texorpdfstring{$G^\circ$}{Go} is a direct sum of copies of \texorpdfstring{$L(\lambda)$}{Llambda}}\label{ss:lambda-E}
%-----------------------------------------------------------------------

We continue with the setting of~\S\ref{ss:twisted}, and in particular with our fixed $\lambda \in \bX^+$.
If $E$ is a finite-dimensional left $\sA^\lambda$-module,
we define a $G^\lambda$-action on the $\bk$-vector space $E \otimes L(\lambda)$
by
\begin{equation}
\label{eqn:L-module}
\iota(a)g \cdot (u \otimes v) = \rho_au \otimes \theta_a^{-1}(gv)
\qquad\text{for $a \in A^\lambda$ and $g \in G^\circ$.}
\end{equation}

\begin{lem}
\label{lem:L-module}
The rule~\eqref{eqn:L-module} defines a structure of $G^\lambda$-module on $E \otimes L(\lambda)$.
\end{lem}

\begin{proof}
Note that
\[
\iota(a)g\iota(b)h = \iota(a)\iota(b) \ad(\iota(b)^{-1})(g) h = \iota(ab) (\gamma(a,b) \ad(\iota(b)^{-1})(g) h).
\]
We now have
\begin{align*}
\iota(a)g\cdot (\iota(b)h \cdot (u \otimes v))
&= \iota(a)g \cdot (\rho_bu \otimes \theta_b^{-1}(hv)) \\
&= \rho_a\rho_b u \otimes \theta_a^{-1}(g \theta_b^{-1}(hv)) \\
&= \alpha(a,b) \rho_{ab} u \otimes (\theta_b \circ \theta_a)^{-1}(\ad(\iota(b)^{-1})(g)hv) \\
&= \rho_{ab} u \otimes \theta_{ab}^{-1}(\gamma(a,b) \ad(\iota(b)^{-1})(g) hv) \\
&= (\iota(a)g\iota(b)h) \cdot (u \otimes v),
\end{align*}
proving the desired formula.
\end{proof}

\begin{prop}
\label{prop:Glambda}
The assignment $E \mapsto E \otimes L(\lambda)$ defines a bijection between the set of isomorphism classes of simple $\sA^\lambda$-modules and the set of isomorphism classes of simple $G^\lambda$-modules whose restriction to $G^\circ$ is a direct sum of copies of $L(\lambda)$.
\end{prop}

\begin{proof}
We will show that if $V$ is a finite dimensional $G^\lambda$-module whose restriction to $G^\circ$ is a direct sum of copies of $L(\lambda)$, and if we set $E:=\Hom_{G^\circ}(L(\lambda),V)$, then $E$ has a natural structure of a left $\sA^\lambda$-module, and there exists an isomorphism of $G^\lambda$-modules
\[
\eta_{\lambda,E} : E \otimes L(\lambda) \simto V.
\]

We define the $\sA^\lambda$-action on $E$ by
\[
(\rho_a \cdot f)(x) = \iota(a) \cdot f(\theta_a(x)) 
\]
for $f \in E=\Hom_{G^\circ}(L(\lambda),V)$ and $x \in L(\lambda)$. (We leave it to the reader to check that $\rho_a \cdot f$ is a morphism of $G^\circ$-modules.) To justify that this defines an $\sA^\lambda$-module structure, we simply compute:
\begin{align*}
(\rho_a \cdot (\rho_b \cdot f))(x) &= \iota(a) \cdot (\rho_b \cdot f)(\theta_a(x)) \\
&= \iota(a) \cdot \iota(b) \cdot f(\theta_b \circ \theta_a(x)) \\
&= \iota(ab) \cdot \gamma(a,b) \cdot f(\theta_b \circ \theta_a(x)) \\
&= \iota(ab) \cdot f(\gamma(a,b) \cdot \theta_b \circ \theta_a(x)) \\
&= \alpha(a,b) \cdot \iota(ab) \cdot f(\theta_{ab}(x)) \\
&= ((\alpha(a,b) \rho_{ab}) \cdot f)(x).
\end{align*}
Now there exists a canonical isomorphism of $G^\circ$-modules
\[
\eta_{\lambda,E} :
E \otimes L(\lambda) = \Hom_{G^\circ}(L(\lambda),V) \otimes L(\lambda) \simto V,
\]
defined by $\eta_{\lambda,E}(f \otimes v) = f(v)$. Let us check that this morphism also commutes with the action of $\iota(A)$.
By definition we have
\[
\iota(a) \cdot (f \otimes v) = (\rho_a \cdot f) \otimes \theta_a^{-1}(v) = \sigma(\iota(a)) \circ f \circ \theta_a \otimes \theta_a^{-1}(v),
\]
where $\sigma : G^\lambda \to \mathrm{GL}(V)$ is the morphism defining the $G^\lambda$-action. Hence
\[
\eta_{\lambda,E} \bigl( \iota(a) \cdot (f \otimes v) \bigr) = \iota(a) \cdot f(v) = \iota(a) \cdot \eta_{\lambda,E}(f \otimes v),
\]
proving that $\eta_{\lambda,E}$ is an isomorphism of $G^\lambda$-modules.

It is clear that the assignments
\[
- \otimes L(\lambda) : E \mapsto E \otimes L(\lambda) \quad \text{and} \quad \Hom_{G^\circ}(L(\lambda),-) : V \mapsto \Hom_{G^\circ}(L(\lambda),V)
\]
define functors from the category of finite-dimensional $\sA^\lambda$-modules to the category of finite-dimensional $G^\lambda$-modules whose restriction to $G^\circ$ are isomorphic to a direct sum of copies of $L(\lambda)$, and from the category of finite-dimensional $G^\lambda$-modules whose restriction to $G^\circ$ are isomorphic to a direct sum of copies of $L(\lambda)$ to the category of finite-dimensional $\sA^\lambda$-modules respectively. It is straightforward to construct an isomorphism of functors $\Hom_{G^\circ}(L(\lambda),-) \circ (- \otimes L(\lambda)) \simto \id$, 
as well as an isomorphism $(- \otimes L(\lambda)) \circ \Hom_{G^\circ}(L(\lambda),-) \simto \id$ defined by $\eta_{\lambda,-}$. Our functors are thus equivalences of categories, quasi-inverse to each other; hence they define bijections between the sets of isomorphism classes of simple objects in these categories.
\end{proof}

\begin{rmk}
\label{rmk:indep-2}
As in Remark~\ref{rmk:indep-1}, it can be easily checked that the assignment $E \mapsto E \otimes L(\lambda)$ does not depend on the choice of the isomorphisms $\{\theta_a : a \in A\}$, in the sense that if $\{\theta'_a : a \in A\}$ is another choice of such isomorphisms, and if $(\sA')^\lambda$ is the corresponding algebra, then the identification $(\sA')^\lambda \simto \sA^\lambda$ considered in Remark~\ref{rmk:indep-1} defines a bijection between isomorphism classes of simple $(\sA')^\lambda$-modules and $\sA^\lambda$-modules, which commutes with the operations $- \otimes L(\lambda)$. Of course, these constructions do not depend on the choice of $L(\lambda)$ in its isomorphism class either.
\end{rmk}

%--------------------------------------------------------------------------
\subsection{Induction from \texorpdfstring{$G^\lambda$}{Glambda} to \texorpdfstring{$G$}{G}}
\label{ss:induction-Glambda}
%--------------------------------------------------------------------------

We continue with the setting of~\S\S\ref{ss:twisted}--\ref{ss:lambda-E}. If $E$ is a finite-dimensional $\mathscr{A}^\lambda$-module, we now consider the $G$-module
\[
L(\lambda,E) := \Ind_{G^\lambda}^G(E \otimes L(\lambda)).
\]

\begin{lem}
\label{lem:Ind-simple}
If $E$ is a simple $\sA^\lambda$-module, then $L(\lambda, E)$ is a simple $G$-module.
\end{lem}

\begin{proof}
Let $V \subset L(\lambda,E)$ be a simple $G$-submodule. For any simple $G^\circ$-module $L$, let $[V:L]_{G^\circ}$ denote the multiplicity of $L$ as a composition factor of $V$, regarded as a $G^\circ$-module.  The image of the embedding $V \hookrightarrow L(\lambda,E)$ under the isomorphism
\[
\Hom_G \bigl( V,  L(\lambda,E) \bigr) = \Hom_G \bigl( V,  \Ind_{G^\lambda}^G(E \otimes L(\lambda)) \bigr) \cong \Hom_{G^\lambda}(V,E \otimes L(\lambda))
\]
given by Frobenius reciprocity provides a nonzero morphism of $G^\lambda$-modules $V \to E \otimes L(\lambda)$, which must be surjective since $E \otimes L(\lambda)$ is simple by Proposition~\ref{prop:Glambda}. It follows that $[V: L(\lambda)]_{G^\circ} \ge \dim(E)$.
Now, as in~\eqref{eqn:induce-restrict}, if $g_1, \ldots, g_r$ are representatives in $G$ of the cosets in $G/G^\lambda$, then as $G^\circ$-modules we have
\[
L(\lambda,E) = \Ind_{G^\lambda}^G(E \otimes L(\lambda)) \cong \bigoplus_{i=1}^r {}^{g_i} \hspace{-1pt} L(\lambda)^{\oplus \dim(E)}.
\]
Since $V$ is stable under the $G$-action, we have $[V: L(\lambda)]_{G^\circ} = [V: {}^{g_i}\hspace{-1pt} L(\lambda)]_{G^\circ}$ for all $i$
(see Lemma~\ref{lem:twist-inner}), and hence $[V : {}^{g_i} \hspace{-1pt} L(\lambda)]_{G^\circ} \ge \dim(E)$ for all $i$.  This implies that $\dim(V) \ge \dim(\Ind_{G^\lambda}^G(E \otimes L(\lambda)))$, so in fact $V = \Ind_{G^\lambda}^G(E \otimes L(\lambda))$, as desired.
\end{proof}

%--------------------------------------------------------------------------
\subsection{Simple \texorpdfstring{$G$}{G}-modules}
%--------------------------------------------------------------------------

We come back to the general setting of~\S\ref{ss:disconn}. (In particular, the dominant weight $\lambda$ is not fixed anymore.)
We can now prove that the procedure explained in~\S\S\ref{ss:twisted}--\ref{ss:induction-Glambda} allows us to construct \emph{all} simple $G$-modules (up to isomorphism).

\begin{lem}
\label{lem:simple-Ind}
Let $V$ be a simple $G$-module. Then there exists $\lambda \in \bXp$, a simple $\sA^\lambda$-module $E$, and an isomorphism of $G$-modules
\[
V \simto L(\lambda,E).
\]
\end{lem}

\begin{proof}
Certainly there exists $\lambda \in \bXp$ and a surjection of $G^\circ$-modules $V \twoheadrightarrow L(\lambda)$. By Frobenius reciprocity we deduce a nonzero (hence injective) morphism of $G$-modules $V \hookrightarrow \Ind_{G^\circ}^{G}(L(\lambda))$. So to conclude, it suffices to prove that all composition factors of $\Ind_{G^\circ}^{G}(L(\lambda))$ are of the form $L(\lambda,E)$ (with $E$ a simple $\sA^\lambda$-module). However, we have
\[
\Ind_{G^\circ}^{G}(L(\lambda)) \cong \Ind_{G^\lambda}^{G} \bigl( \Ind_{G^\circ}^{G^\lambda}(L(\lambda)) \bigr).
\]
The restriction of $\Ind_{G^\circ}^{G^\lambda}(L(\lambda))$ to $G^\circ$ is a direct sum of copies of $L(\lambda)$ by~\eqref{eqn:induce-restrict} applied to $G^\lambda$. Therefore, all of its composition factors are of the form $E \otimes L(\lambda)$ with $E$ a simple $\sA^\lambda$-module by Proposition~\ref{prop:Glambda}. Since the functor $\Ind_{G^\lambda}^{G}$ is exact (by Lemma~\ref{lem:ind-tensor}, or by~\cite[Corollary~I.5.13]{jantzen}) and sends simple $G^\lambda$-modules of the form $E \otimes L(\lambda)$ to simple $G$-modules by Lemma~\ref{lem:Ind-simple}, the claim follows.
\end{proof}

%--------------------------------------------------------------------------
\subsection{Conjugation}
\label{ss:conj}
%--------------------------------------------------------------------------

It now remains to understand when two modules of the form $L(\lambda,E)$ are isomorphic. For this, we need to analyze the relation between this construction applied to a dominant weight, and to a twist of this dominant weight by an element of $A$.

So, let $\lambda \in \bXp$, and $a \in A$. Then we have
\[
A^{{}^a \hspace{-1pt} \lambda} = a A^\lambda a^{-1}, \quad G^{{}^a \hspace{-1pt} \lambda} = \iota(a) G^\lambda \iota(a)^{-1},
\]
and we can choose as $L({}^a \hspace{-1pt} \lambda)$ the module ${}^{\iota(a)} \hspace{-1pt} L(\lambda)$, cf.~\eqref{eqn:twist-standard}.

Let us choose isomorphisms $\theta_b : L(\lambda) \simto {}^{\iota(b)} \hspace{-1pt} L(\lambda)$ for all $b \in A^\lambda$. Again for $b \in A^\lambda$, we can consider the isomorphism
\begin{multline*}
\tilde{\theta}_{aba^{-1}} : L({}^a \hspace{-1pt} \lambda) 
=
{}^{\iota(a)} \hspace{-1pt} L(\lambda) \xrightarrow{\theta_b} {}^{\iota(a) \iota(b)} \hspace{-1pt} L(\lambda) = {}^{\iota(a) \iota(b) \iota(a)^{-1}} \hspace{-1pt} \bigl({}^{\iota(a)} \hspace{-1pt} L(\lambda) \bigr) \\
= {}^{\iota(a)\iota(b)\iota(a)^{-1}} \hspace{-1pt} \bigl(L({}^{a} \hspace{-1pt} \lambda) \bigr)
\xrightarrow[\sim]{\iota(aba^{-1})^{-1} \iota(a) \iota(b) \iota(a)^{-1} \cdot (-)} {}^{\iota(aba^{-1})} \hspace{-1pt} L({}^a \hspace{-1pt} \lambda).
\end{multline*}
(Here, the last isomorphism means the action of $\iota(aba^{-1})^{-1} \iota(a) \iota(b) \iota(a)^{-1}$ on $L({}^a \hspace{-1pt} \lambda)$, or in other words the action of $\iota(a)^{-1} \iota(aba^{-1})^{-1} \iota(a) \iota(b)$ on $L(\lambda)$.) 

The following claim can be checked directly from the definitions.

\begin{lem}
\label{lem:cocycle-conjugation}
For any $b,c \in A^\lambda$, we have
\[
\gamma(aca^{-1},aba^{-1}) \circ \tilde{\theta}_{aba^{-1}} \circ \tilde{\theta}_{aca^{-1}} = \alpha(c,b) \cdot \tilde{\theta}_{acba^{-1}}
\]
(where here $\gamma(aca^{-1},aba^{-1})$ means the action of this element on $L({}^a \hspace{-1pt} \lambda)$).
\end{lem}

If $\sA^\lambda$ and its basis $\{\rho_b : b \in A^\lambda\}$ are defined in terms of the isomorphisms $\{\theta_b : b \in A^\lambda\}$ and if $\sA^{{}^a \hspace{-1pt} \lambda}$ and its basis $\{\tilde{\rho}_b : b \in A^{{}^a \hspace{-1pt} \lambda}\}$ are defined in terms of the isomorphisms $\{\tilde{\theta}_a : a \in A^{{}^a \hspace{-1pt} \lambda}\}$, then Lemma~\ref{lem:cocycle-conjugation} allows us to compare the cocycles that arise in the definitions of $\sA^\lambda$ and $\sA^{{}^a \hspace{-1pt} \lambda}$. More precisely, this lemma shows that the assignment $\rho_b \mapsto \tilde{\rho}_{aba^{-1}}$ defines an algebra isomorphism $\xi_\lambda^a : \sA^\lambda \simto \sA^{{}^a \hspace{-1pt} \lambda}$. 

The isomorphism $\xi_\lambda^a$ can be described more canonically as follows. Recall that Proposition~\ref{prop:end-induced} provides canonical identifications
\[
(\sA^\lambda)^\op \simto \End_{G^\lambda} \bigl( \Ind_{G^\circ}^{G^\lambda} (L(\lambda)) \bigr), \quad (\sA^{{}^a \hspace{-1pt} \lambda})^\op \simto \End_{G^\lambda} \bigl( \Ind_{G^\circ}^{G^\lambda} (L({}^a \hspace{-1pt} \lambda)) \bigr).
\]
One can check that
under these identifications, the automorphism $\xi^a_\lambda$ is given by the isomorphism
\[
\End_{G^\lambda} \bigl( \Ind_{G^\circ}^{G^\lambda} (L(\lambda)) \bigr) = \End_{G^{ {}^a \hspace{-1pt} \lambda}} \bigl( {}^{\iota(a)} \hspace{-1pt} \Ind_{G^\circ}^{G^\lambda} (L(\lambda)) \bigr) \simto \End_{G^{ {}^a \hspace{-1pt} \lambda}} \bigl( \Ind_{G^\circ}^{G^{{}^a \hspace{-1pt} \lambda}} (L({}^a \hspace{-1pt} \lambda)) \bigr)
\]
(where we use the notation of Remark~{\rm \ref{rmk:twist-Ind}}).

The properties of these isomorphisms that we will need below are summarized in the following lemma.

\begin{lem}
\label{lem:properties-xi}
Let $\lambda \in \bX^+$.
\begin{enumerate}
\item
\label{it:isom-product}
If $a,b \in A$, then we have $\xi_\lambda^{ab} = \xi_{{}^b \hspace{-1pt} \lambda}^a \circ \xi^b_\lambda$.
\item
\label{it:isom-inner}
If $a \in A^\lambda$, then $\xi_\lambda^a$ is an inner automorphism of $\sA^\lambda$.
\end{enumerate}
\end{lem}

\begin{proof}
\eqref{it:isom-product}
To simplify notation, we set $\mu := {}^{ab} \hspace{-1pt} \lambda$.
Note that the simple $G^\circ$-modules of highest weight $\mu$ used in the definitions of $\xi_\lambda^{ab}$ and $\xi_{{}^b \hspace{-1pt} \lambda}^a \circ \xi^b_\lambda$ are different: for the former we use the module $L_1(\mu) := {}^{\iota(ab)} \hspace{-1pt} L(\lambda)$, while for the latter we use the module $L_2(\mu) := {}^{\iota(a) \iota(b)} \hspace{-1pt} L(\lambda)$. There exists a canonical isomorphism
\begin{equation}
\label{eqn:isom-algebras-product}
L_1(\mu) \simto L_2(\mu),
\end{equation}
given by the action of $\gamma(a,b)^{-1}$ on $L(\lambda)$ (i.e.~the inverse of the isomorphism denoted $\phi_{a,b}$ in~\S\ref{ss:component}).

Our algebras are all defined as endomorphisms of some induced module, which can be described in terms of functions with values in the vector space underlying the representation $L(\lambda)$. From this point of view, $\xi_{{}^b \hspace{-1pt} \lambda}^a \circ \xi^b_\lambda$ is conjugation by the isomorphism of vector spaces $\Ind_{G^\circ}^{G^\lambda}(L(\lambda)) \simto \Ind_{G^\circ}^{G^\mu}(L_2(\mu))$ sending functions $G^\lambda \to L(\lambda)$ to functions $G^\mu \to L(\lambda)$ and given by $\phi \mapsto \phi(\iota(b)^{-1} \iota(a)^{-1} (-) \iota(a) \iota(b))$, while $\xi_\lambda^{ab}$ is conjugation by the isomorphism $\Ind_{G^\circ}^{G^\lambda}(L(\lambda)) \simto \Ind_{G^\circ}^{G^\mu}(L_1(\mu))$ given by $\phi \mapsto \phi(\iota(ab) (-) \iota(ab)^{-1})$. Taking into account the isomorphism~\eqref{eqn:isom-algebras-product}, we have to check that conjugation by the isomorphism given by
\begin{equation}\label{eqn:isom-algebras-conj1}
\phi \mapsto \gamma(a,b) \cdot \phi(\iota(b)^{-1} \iota(a)^{-1} (-) \iota(a) \iota(b))
\end{equation}
(where $\gamma(a,b) \cdot (-)$ means the action of $\gamma(a,b) \in G^\circ$ on $L(\lambda)$) coincides with conjugation by the isomorphism given by
\begin{equation}\label{eqn:isom-algebras-conj2}
\phi \mapsto \phi(\iota(ab) (-) \iota(ab)^{-1}).
\end{equation}
However, since $\gamma(a,b)$ belongs to $G^\circ$, the functions $\phi$ we consider satisfy
\begin{multline*}
\gamma(a,b) \cdot \phi(\iota(b)^{-1} \iota(a)^{-1} (-) \iota(a) \iota(b)) = \phi(\iota(b)^{-1} \iota(a)^{-1} (-) \iota(a) \iota(b) \gamma(a,b)^{-1}) \\
= \phi(\iota(b)^{-1} \iota(a)^{-1} (-) \iota(ab)) = \phi(\gamma(a,b)^{-1} \iota(ab)^{-1} (-) \iota(ab)).
\end{multline*}
Thus, the isomorphisms~\eqref{eqn:isom-algebras-conj1} and~\eqref{eqn:isom-algebras-conj2} do \emph{not} coincide, but they differ only by the action of an element of $G^\lambda$ (which, in fact, even belongs to $G^\circ$) on $\Ind_{G^\circ}^{G^\lambda} (L(\lambda))$. Therefore, conjugation by either~\eqref{eqn:isom-algebras-conj1} or~\eqref{eqn:isom-algebras-conj2} induces the \emph{same} isomorphism of algebras $\End_{G^\lambda}(\Ind_{G^\circ}^{G^\lambda}(L(\lambda))) \simto \End_{G^\mu} ( \Ind_{G^\circ}^{G^\mu} (L_1(\mu)) )$.

\eqref{it:isom-inner}
By the comments preceding the statement, $\xi^a_\lambda$ is conjugation by an isomorphism $\Ind_{G^\circ}^{G^\lambda} (L(\lambda)) \simto \Ind_{G^\circ}^{G^{{}^a \hspace{-1pt} \lambda}} (L({}^a \hspace{-1pt} \lambda))$. If $a \in A^\lambda$ then this isomorphism defines an invertible element of $\sA^\lambda$, so that $\xi^a_\lambda$ is indeed an inner automorphism.
\end{proof}

Given $a \in A$ and $\lambda \in \bXp$, the isomorphism $\xi^a_\lambda$ defines a bijection between the set of isomorphism classes of simple $\sA^\lambda$-modules and the set of isomorphism classes of simple $\sA^{{}^a \hspace{-1pt} \lambda}$-modules.
From Lemma~\ref{lem:properties-xi}\eqref{it:isom-product} we see
that this operation
defines an action of the group $A$ on the set of pairs $(\lambda, E)$ where $\lambda \in \bXp$ and $E$ is a simple $\sA^\lambda$-module.
Moreover,
it follows from Lemma~\ref{lem:properties-xi}\eqref{it:isom-inner} that the induced
action of $A^\lambda$ on the set of isomorphism classes of simple $\sA^\lambda$-modules is trivial.

\begin{lem}
\label{lem:conjugation-isom}
Let $\lambda \in \bXp$, and let $E$ be a simple $\sA^\lambda$-module. Let $a \in A$, and let $E'$ be the simple $\sA^{{}^a \hspace{-1pt} \lambda}$-module deduced from $E$ via the isomorphism $\xi^a_\lambda : \sA^\lambda \simto \sA^{{}^a \hspace{-1pt} \lambda}$. Then there exists an isomorphism of $G$-modules
\[
L(\lambda,E) \simto L({}^a \hspace{-1pt} \lambda, E').
\]
\end{lem}

\begin{proof}
As above we choose for our simple $G^\circ$-module of highest weight ${}^a \hspace{-1pt} \lambda$ the module ${}^{\iota(a)} \hspace{-1pt} L(\lambda)$. Then conjugation by $\iota(a)$ induces an isomorphism $G^\lambda \simto G^{{}^a \hspace{-1pt} \lambda}$, and using the notation of Remark~\ref{rmk:twist-Ind} we have as $G^{{}^a \hspace{-1pt} \lambda}$-modules
\[
{}^{\iota(a)} \hspace{-1pt} (E \otimes L(\lambda)) = E' \otimes L({}^a \hspace{-1pt} \lambda).
\]
In view of Lemma~\ref{lem:twist-Ind} we deduce an isomorphism of $G$-modules
\[
{}^{\iota(a)} \hspace{-1pt} \Ind_{G^\lambda}^G(E \otimes L(\lambda)) \simto \Ind_{G^{{}^a \hspace{-1pt} \lambda}}^G(E' \otimes L({}^a \hspace{-1pt} \lambda)).
\]
Now by Lemma~\ref{lem:twist-inner} the left-hand side is isomorphic to $L(\lambda,E)$, and the claim follows.
\end{proof}

%--------------------------------------------------------------------------
\subsection{Classification of simple \texorpdfstring{$G$}{G}-modules}
%--------------------------------------------------------------------------

We denote by $\Irr(G)$ the set of isomorphism classes of simple $G$-modules. Now we can finally state the main result of this section.

\begin{thm}
\label{thm:disconn-class}
The assignment $(\lambda,E) \mapsto L(\lambda,E)$ induces a bijection
\[
\left.
\left\{
(\lambda, E) \, \left|\,
\begin{array}{c}
\text{$\lambda \in \bXp$ and $E$ an isom. class} \\
\text{of simple left $\sA^\lambda$-modules}
\end{array}
\right\} \right.
\right/ A \longleftrightarrow \Irr(G).
\]
\end{thm}

\begin{proof}
From Lemma~\ref{lem:Ind-simple}, we see that the assignment $(\lambda,E) \mapsto L(\lambda,E)$ defines a map from the set of pairs $(\lambda,E)$ as in the statement to the set $\Irr(G)$. By Lemma~\ref{lem:conjugation-isom} this map factors through a map
\[
\left.
\left\{
(\lambda, E) \, \left|\,
\begin{array}{c}
\text{$\lambda \in \bXp$ and $E$ an isom. class} \\
\text{of simple left $\sA^\lambda$-modules}
\end{array}
\right\} \right.
\right/ A \to \Irr(G).
\]
By Lemma~\ref{lem:simple-Ind}, this latter map is surjective. Hence, all that remains is to prove that it is injective.

Let $(\lambda,E)$ and $(\lambda',E')$ be pairs as above.
Let $V = L(\lambda,E)$ and $V' = L(\lambda',E')$, and assume that $V \cong V'$. As a $G^\circ$-representation, $V$ is isomorphic to a direct sum of twists of $L(\lambda)$, and $V'$ is isomorphic to a direct sum of twists of $L(\lambda')$ (see the proof of Lemma~\ref{lem:Ind-simple}). Hence $L(\lambda)$ and $L(\lambda')$ are twists of each other, which implies that $\lambda$ and $\lambda'$ are in the same $A$-orbit. Therefore, we can (and shall) assume that $\lambda=\lambda'$. Fix some isomorphism $V \simto V'$, and consider the morphism of $G^\lambda$-modules $f : V \to E' \otimes L(\lambda)$ deduced by Frobenius reciprocity. 
If $g_1, \ldots, g_r$ are representatives of the cosets in $G/G^\lambda$, with $g_1=1_G$, then we have an isomorphism of $G^\circ$-modules
\[
 \Ind_{G^\lambda}^G (E \otimes L(\lambda)) \cong \bigoplus_{i=1}^r {}^{g_i} \hspace{-1pt} L(\lambda) \otimes E.
\]
If $i \neq 1$, then ${}^{g_i} \hspace{-1pt} L(\lambda)$ is not isomorphic to $L(\lambda)$. Hence $f$ is zero on the corresponding summand of $\Ind_{G^\lambda}^G (E \otimes L(\lambda))$. We deduce that the composition
\[
E \otimes L(\lambda) \hookrightarrow \Ind_{G^\lambda}^G (E \otimes L(\lambda)) \xrightarrow{f} E' \otimes L(\lambda),
\]
where the first morphism is again deduced from Frobenius reciprocity,
is nonzero. But this morphism is a morphism of $G^\lambda$-modules. Since $L(\lambda,E)$ and $L(\lambda,E')$ are simple, it must be an isomorphism, and by Proposition~\ref{prop:Glambda} this implies that $E \cong E'$ as $\sA^\lambda$-modules.
\end{proof}

\begin{rmk}\phantomsection
\label{rmk:classification-simples}
\begin{enumerate}
\item
\label{it:rmks-representatives}
As explained above Lemma~\ref{lem:conjugation-isom}, for any $\lambda \in \bXp$ the action of $A^\lambda$ on the set of isomorphism classes of irreducible $\sA^\lambda$-modules is trivial. Hence if $\Lambda \subset \bXp$ is a set of representatives of the $A$-orbits in $\bXp$, the quotient considered in the statement of Theorem~\ref{thm:disconn-class} can be described more explicitly as the set of pairs $(\lambda,E)$ where $\lambda \in \Lambda$ and $E$ is an isomorphism class of simple $\sA^\lambda$-modules.
\item
Assume that $\iota$ is a group morphism (so that $G$ is isomorphic to the semi-direct product $A \ltimes G^\circ$) and that moreover there exists a Borel subgroup $B \subset G^\circ$ such that $\iota(a)B\iota(a)^{-1}=B$ for any $a \in A$. Then if we define the standard and costandard $G^\circ$-modules using this Borel subgroup, the isomorphisms
\[
{}^{\iota(a)} \hspace{-1pt} \Delta(\lambda) \cong \Delta({}^a \hspace{-1pt} \lambda), \quad {}^{\iota(a)} \hspace{-1pt} \nabla(\lambda) \cong \nabla({}^a \hspace{-1pt} \lambda)
\]
(see~\eqref{eqn:twist-standard}) can be chosen in a canonical way. In fact, our assumptions imply that there exist unique $B$-stable lines in ${}^{\iota(a)} \hspace{-1pt} \Delta(\lambda)$ and $\Delta({}^a \hspace{-1pt} \lambda)$, and moreover that these lines coincide. Hence there exists a unique isomorphism of $G$-modules ${}^{\iota(a)} \hspace{-1pt} \Delta(\lambda) \simto \Delta({}^a \hspace{-1pt} \lambda)$ which restricts to the identity on these $B$-stable lines. Similar comments apply to ${}^{\iota(a)} \hspace{-1pt} \nabla(\lambda)$ and $\nabla({}^a \hspace{-1pt} \lambda)$.

In particular, the isomorphisms $\theta_a$ of~\S\ref{ss:twisted} can be chosen in a canonical way. Then the cocycle $\alpha$ will be trivial, so that in this case $\sA^\lambda$ is canonically isomorphic to the group algebra $\bk[A^\lambda]$.
\end{enumerate}
\end{rmk}

%--------------------------------------------------------------------
\subsection{Semisimplicity}
%--------------------------------------------------------------------

We finish this section with a criterion ensuring that the algebra $\sA^\lambda$ is semisimple unless $p$ is small.

\begin{lem}\label{lem:induced-semisimple}
Assume that $p \nmid |A|$.
If $V$ is a simple $G^\circ$-module, then $\Ind_{G^\circ}^G (V)$ is a semisimple $G$-module.
\end{lem}

\begin{proof}
Let $M$ be a $G$-submodule of $\Ind_{G^\circ}^G (V)$, and let $N = \Ind_{G^\circ}^G (V)/M$. We will show that the image $c$ of the exact sequence $M \hookrightarrow \Ind_{G^\circ}^G (V) \twoheadrightarrow N$ in $\Ext^1_G(N,M)$ vanishes.

First we remark that for any two algebraic $G$-modules $X,Y$, the forgetful functor from $\Rep(G)$ to $\Rep(G^\circ)$ induces an isomorphism
\[
\Hom_G(X,Y) \simto \bigl( \Hom_{G^\circ}(X,Y) \bigr)^A.
\]
Under our assumptions the functor $(-)^A$ is exact. On the other hand, it is easily checked that the restriction of any injective $G$-module to $G^\circ$ is injective. Hence this isomorphism induces an isomorphism
\[
\Ext^n_G(X,Y) \simto \bigl( \Ext^n_{G^\circ}(X,Y) \bigr)^A
\]
for any $n \geq 0$. We deduce in particular that the forgetful functor induces an injection
\[
\Ext^1_G(N,M) \hookrightarrow \Ext^1_{G^\circ}(N,M).
\]
Hence to prove that $c=0$ it suffices to prove that the sequence $M \hookrightarrow \Ind_{G^\circ}^G (V) \twoheadrightarrow N$, considered as an exact sequence of $G^\circ$-modules, splits. This fact is clear since $\Ind_{G^\circ}^G (V)$ is semisimple as a $G^\circ$-module, see~\eqref{eqn:induce-restrict}.
\end{proof}

From this lemma (applied to the group $A^\lambda$) and Proposition~\ref{prop:end-induced} we deduce the following.

\begin{lem}\label{lem:semisimple-algebra}
If $p \nmid | A^\lambda|$, then
the algebra $\sA^\lambda$ is semisimple (and in fact isomorphic to a product of matrix algebras).
\end{lem}

%%%%%%%%%%%%%%%%%%%%%%%%%%%%%%%%%%%%%%%%%%%%%%%%%%%%%%%%%%%
\section{Highest weight structure}
\label{sec:disconn-hwt}
%%%%%%%%%%%%%%%%%%%%%%%%%%%%%%%%%%%%%%%%%%%%%%%%%%%%%%%%%%%

Our goal in this section is to prove that if 
$p \nmid |A|$, then the category $\Rep(G)$ of finite-dimensional $G$-modules admits a natural structure of a highest weight category.

For the beginning of the section, we continue with the setting of~\S\ref{ss:disconn} (not imposing any further assumption).

%--------------------------------------------------------------------------
\subsection{The order}
%--------------------------------------------------------------------------

If $(\lambda,E)$ is a pair as in Theorem~\ref{thm:disconn-class}, we denote by $[\lambda,E]$ the corresponding $A$-orbit.
We define a relation $<$ on the set of such orbits as follows:
\begin{equation}
\label{eqn:def-order}
[\lambda,E] < [\lambda',E']
\qquad\text{if}\qquad
\text{for some $a \in A$, we have ${}^a \hspace{-1pt} \lambda < \lambda'$.}
\end{equation}
(Here, the order on $\bX$ is the standard one, where $\lambda \leq \mu$ iff $\mu - \lambda$ is a sum of positive roots.)

\begin{lem}
The relation $<$ is a partial order.
\end{lem}

\begin{proof}
Using the fact that for $a \in A$ and $\lambda, \mu \in \bX$ such that $\lambda \leq \mu$ we have ${}^a \hspace{-1pt} \lambda \leq {}^a \hspace{-1pt} \mu$ (because the $A$-action is linear and preserves positive roots), one can easily check that this relation is transitive. What remains to be seen is that there cannot exist classes $[\lambda,E]$, $[\lambda',E']$ such that
\[
[\lambda,E] < [\lambda',E'] < [\lambda,E].
\]
However, in this case we have ${}^a \hspace{-1pt} \lambda < \lambda$ for some $a \in A$. Since $a$ permutes the positive coroots of $G^\circ$, then
 if we denote by $2\rho^\vee$ the sum of these coroots we must have $\langle {}^a \hspace{-1pt} \lambda, 2\rho^\vee \rangle = \langle \lambda, 2\rho^\vee \rangle$, hence $\langle \lambda - {}^a \hspace{-1pt} \lambda, 2\rho^\vee \rangle=0$. On the other hand, by assumption $\lambda - {}^a \hspace{-1pt} \lambda$ is a nonzero sum of positive roots, so that its pairing with $2\rho^\vee$ cannot vanish. This provides the desired contradiction.
\end{proof}

%--------------------------------------------------------------------------
\subsection{Standard \texorpdfstring{$G$}{G}-modules}
\label{ss:standard}
%--------------------------------------------------------------------------

Let $\lambda \in \bXp$.  We will work in the setting of \S\S\ref{ss:component}--\ref{ss:twisted}, including, in particular, fixing a $G^\circ$-module $L(\lambda)$, and notation such as $\iota$, $\gamma$, $\theta$, and $\alpha$. We also fix a representative $\Delta(\lambda)$ for the Weyl module surjecting to $L(\lambda)$, and a surjection $\pi^\lambda : \Delta(\lambda) \to L(\lambda)$.

Since $\End_{G^\circ}(\Delta(\lambda))=\bk \cdot \id$, from~\eqref{eqn:twist-standard} we see that
for each $a \in A^\lambda$, there exists a unique isomorphism $\theta^\Delta_a: \Delta(\lambda) \simto {}^{\iota(a)} \hspace{-1pt} \Delta(\lambda)$ such that the following diagram commutes:
\[
\begin{tikzcd}
\Delta(\lambda) \ar[r, "\theta^\Delta_a"] \ar[d, two heads] &
{}^{\iota(a)} \hspace{-1pt} \Delta(\lambda) \ar[d, two heads] \\
L(\lambda) \ar[r, "\theta_a"] & {}^{\iota(a)} \hspace{-1pt} L(\lambda).
\end{tikzcd}
\]
Moreover, this uniqueness implies that
for any $a, b \in A^\lambda$, if we define $\phi^\Delta_{a,b} : \Delta(\lambda) \to \Delta(\lambda)$ as the action of $\gamma(a,b)$, then we have
\begin{equation}
\label{eqn:delta-cocyle}
\phi^\Delta_{a,b}\theta^\Delta_b \theta^\Delta_a = \alpha(a,b) \theta^\Delta_{ab}.
\end{equation}

\begin{rmk}
\label{rmk:Alambda-standards}
These considerations show that the subgroup $A^\lambda \subset A$ can be equivalently defined as consisting of the elements $a \in A$ such that  ${}^{\iota(a)} \hspace{-1pt} \Delta(\lambda) \cong \Delta(\lambda)$. The twisted group algebra $\mathscr{A}^\lambda$ can also be defined in terms of a choice of isomorphisms $(\theta^\Delta_a : a \in A^\lambda)$ instead of isomorphisms $(\theta_a : a \in A^\lambda)$.
\end{rmk}

\begin{lem}
\label{lem:Delta-disconn}
Let $E$ be a finite-dimensional left $\sA^\lambda$-module. 
The following rule defines the structure of a $G^\lambda$-module on the vector space $E \otimes \Delta(\lambda)$:
\[
\iota(a)g \cdot (u \otimes v) = \rho_au \otimes (\theta_a^\Delta)^{-1}(gv)
\qquad\text{for any $a \in A^\lambda$ and $g \in G^\circ$.}
\]
If $E$ is simple, this $G^\lambda$-module has $E \otimes L(\lambda)$ as its unique irreducible quotient.  Moreover, all the $G^\circ$-composition factors of the kernel of the quotient map $E \otimes \Delta(\lambda) \to E \otimes L(\lambda)$ are of the form $L(\mu)$ with $\mu < \lambda$.
\end{lem}

\begin{proof}
We begin by noting that thanks to~\eqref{eqn:delta-cocyle},
the calculation from Lemma~\ref{lem:L-module} can be repeated to show that the formula above does, indeed, define the structure of a $G^\lambda$-module on $E \otimes \Delta(\lambda)$.  Moreover, the quotient map $\pi^\lambda: \Delta(\lambda) \to L(\lambda)$ induces a surjective map of $G^\lambda$-modules
\[
\pi^\lambda_E := \id_E \otimes \pi: E \otimes \Delta(\lambda) \to E \otimes L(\lambda).
\]

Now, assume that $E$ is simple.
If we forget the $G^\lambda$-module structure and regard $E \otimes \Delta(\lambda)$ as just a $G^\circ$-module, then it is clear that its unique maximal semisimple quotient can be identified with $E \otimes L(\lambda)$, and that the highest weights of the kernel of $\pi^\lambda_E$ are${}< \lambda$.
Let $M$ be the head of $E \otimes \Delta(\lambda)$ as a $G^\lambda$-module.  Since $M$ must remain semisimple as a $G^\circ$-module (by Lemma~\ref{lem:restrict-semis}), it cannot be larger than $E \otimes L(\lambda)$.  In other words, $E \otimes L(\lambda)$ is the unique simple quotient of $E \otimes \Delta(\lambda)$.
\end{proof}

\begin{prop}
\label{prop:standard-disconn}
Let $E$ be a simple $\sA^\lambda$-module. The $G$-module
\[
\Delta(\lambda,E) :=
\Ind_{G^\lambda}^G(E \otimes \Delta(\lambda))
\]
admits $L(\lambda,E)$ as its unique irreducible quotient. Moreover, all the composition factors of the kernel of the quotient map $\Delta(\lambda,E) \to L(\lambda,E)$ are of the form $L(\mu,E')$ with $[\mu,E'] < [\lambda,E]$.
\end{prop}

\begin{proof}
The surjection $E \otimes \Delta(\lambda) \to E \otimes L(\lambda)$ from Lemma~\ref{lem:Delta-disconn} induces a surjection $\Delta(\lambda,E) \to L(\lambda,E)$ since the functor $\Ind_{G^\lambda}^G$ is exact (see the proof of Lemma~\ref{lem:simple-Ind}). If $g_1, \cdots, g_r$ are representatives of the cosets in $G/G^\lambda$, then as $G^\circ$-modules we have
\begin{equation}
\label{eqn:isom-Delta-G0}
\Delta(\lambda,E) \cong \bigoplus_{i=1}^r E \otimes {}^{g_i} \hspace{-1pt} \Delta(\lambda), \quad L(\lambda,E) \cong \bigoplus_{i=1}^r E \otimes {}^{g_i} \hspace{-1pt} L(\lambda).
\end{equation}
Therefore, as in the proof of Lemma~\ref{lem:simple-Ind}, $L(\lambda,E)$ is the head of $\Delta(\lambda,E)$ as a $G^\circ$-module, hence also as a $G$-module.

If $L(\mu,E')$ is a $G$-composition factor of the kernel of the surjection $\Delta(\lambda,E) \to L(\lambda,E)$, then some twist of $L(\mu)$ must be a $G^\circ$-composition factor of the surjection $ {}^{g_i} \hspace{-1pt} \Delta(\lambda) \to  {}^{g_i} \hspace{-1pt} L(\lambda)$ for some $i$. Therefore $\mu$ is smaller than some twist of $\lambda$, and we deduce that $[\mu,E'] < [\lambda,E]$.
\end{proof}

%--------------------------------------------------------------------------
\subsection{\texorpdfstring{$\Ext^1$}{Ext1}-vanishing}
\label{ss:vanishing}
%--------------------------------------------------------------------------

The same 
proof as for Lemma~\ref{lem:conjugation-isom}
shows that, up to isomorphism, $\Delta(\lambda,E)$ only depends on the orbit $[\lambda,E]$. The following lemma shows that this module is a ``partial projective cover'' of $L(\lambda,E)$ (under the assumption that $p \nmid |A|$).

\begin{lem}
\label{lem:Ext1-Delta}
Assume that $p \nmid |A|$.
For any two pairs $(\lambda,E)$ and $(\mu,E')$, we have
\[
\Ext^1_G \bigl( \Delta(\lambda,E), L(\mu,E') \bigr) \neq 0 \quad \Rightarrow \quad [\mu, E'] > [\lambda,E].
\]
\end{lem}

\begin{proof}
As in the proof of Lemma~\ref{lem:induced-semisimple}, we have a canonical isomorphism
\[
\Ext^1_G \bigl( \Delta(\lambda,E), L(\mu,E') \bigr) \cong
\Bigl( \Ext^1_{G^\circ} \bigl( \Delta(\lambda,E), L(\mu,E') \bigr) \Bigr)^A.
\]
If we assume that $\Ext^1_G \bigl(\Delta(\lambda,E), L(\mu,E') \bigr) \neq 0$, then this isomorphism shows that we must also have $\Ext^1_{G^\circ} \bigl( \Delta(\lambda,E), L(\mu,E') \bigr) \neq 0$. Using~\eqref{eqn:isom-Delta-G0}, we deduce that for some $g,h \in G$ we have
\[
\Ext^1_{G^\circ}({}^g \hspace{-1pt} \Delta(\lambda), {}^h \hspace{-1pt} L(\mu)) \neq 0.
\]
This implies that ${}^{\varpi(h)} \hspace{-1pt} \mu > {}^{\varpi(g)} \hspace{-1pt} \lambda$, hence that $[\mu, E'] > [\lambda,E]$.
\end{proof}

%--------------------------------------------------------------------------
\subsection{Costandard \texorpdfstring{$G$}{G}-modules}
\label{ss:costandard}
%--------------------------------------------------------------------------

Fix again $\lambda \in \bXp$ and a simple $\sA^\lambda$-module $E$. Then after fixing a costandard module $\nabla(\lambda)$ with socle $L(\lambda)$ and an embedding $L(\lambda) \hookrightarrow \nabla(\lambda)$, as in~\S\ref{ss:standard} the isomorphisms $\theta_a$ can be ``lifted'' to isomorphisms $\theta_a^\nabla : \nabla(\lambda) \simto {}^{\iota(a)} \hspace{-1pt} \nabla(\lambda)$, which satisfy the appropriate analogue of~\eqref{eqn:delta-cocyle}. Using these isomorphisms one can define a $G^\lambda$-module structure on $E \otimes \nabla(\lambda)$ by the same procedure as in Lemma~\ref{lem:Delta-disconn}. Then the same arguments as for Proposition~\ref{prop:standard-disconn} show that $\nabla(\lambda,E):=\Ind_{G^\lambda}^G(E \otimes \nabla(\lambda))$ admits $L(\lambda,E)$ as its unique simple submodule, and that all the composition factors of the injection $L(\lambda,E) \hookrightarrow \nabla(\lambda,E)$ are of the form $L(\mu,E')$ with $[\mu,E'] < [\lambda,E]$. Moreover, as in Lemma~\ref{lem:Ext1-Delta}, if $p \nmid |A|$ we have
\[
\Ext^1_G \bigl( L(\mu,E'), \nabla(\lambda,E) \bigr) \neq 0 \quad \Rightarrow \quad [\mu, E'] > [\lambda,E].
\]

\begin{lem}
\label{lem:Ext-vanishing}
Assume that $p \nmid |A|$, and
let $(\lambda,E)$ and $(\mu, E')$ be pairs as above. Then for any $i>0$ we have
\[
\Ext^i_G(\Delta(\lambda,E), \nabla(\mu,E'))=0.
\]
Moreover
\[
\Hom_G(\Delta(\lambda,E), \nabla(\mu,E'))=0
\]
unless $[\lambda,E]=[\mu,E']$, in which case this space is $1$-dimensional.
\end{lem}

\begin{proof}
As in the proof of Lemma~\ref{lem:induced-semisimple}, for any $i>0$ we have
\[
\Ext_G^{i}(\Delta(\lambda,E), \nabla(\mu,E')) \cong \\
\bigl( \Ext_{G^\circ}^i(\Delta(\lambda,E), \nabla(\mu,E')) \bigr)^A.
\]
As $G^\circ$-modules $\Delta(\lambda,E)$ is isomorphic to a direct sum of Weyl modules, and $\nabla(\mu,E')$ is isomorphic to a direct sum of induced modules. Hence, the right-hand side vanishes unless $i=0$, which proves the first claim.

For the second claim we remark that if
$\Hom_G(\Delta(\lambda,E), \nabla(\mu,E')) \neq 0$,
then $L(\lambda,E)$ is a composition factor of $\nabla(\mu,E')$, so that $[\lambda,E] \leq [\mu,E']$, and $L(\mu,E')$ is a composition factor of $\Delta(\lambda,E)$, so that $[\mu,E'] \leq [\lambda,E]$. We deduce that $[\mu,E']=[\lambda,E]$. Moreover, in this case any nonzero morphism in this space must be a multiple of the composition
\[
\Delta(\lambda,E) \twoheadrightarrow L(\lambda,E) \hookrightarrow \nabla(\lambda,E),
\]
which concludes the proof.
\end{proof}

%--------------------------------------------------------------------------
\subsection{Highest weight structure}
\label{ss:hw}
%--------------------------------------------------------------------------

Let $\mathscr{C}$ be a finite-length $\bk$-linear abelian category such that $\Hom_{\mathscr{C}}(M,N)$ is finite-dimensional for any $M$, $N$ in $\mathscr{C}$.
Let $\mathscr{S}$ be the set of isomorphism classes of irreducible objects of $\mathscr{C}$. Assume that $\mathscr{S}$ is equipped with a partial order $\leq$, and that for each $s \in \mathscr{S}$ we have a fixed representative of the simple object $L_s$. Assume also we are given, for any $s \in \mathscr{S}$, objects $\Delta_s$ and $\nabla_s$, and morphisms $\Delta_s \to L_s$ and $L_s \to \nabla_s$. For $\mathscr{T} \subset \mathscr{S}$, we denote by $\mathscr{C}_{\mathscr{T}}$ the Serre subcategory of $\mathscr{C}$ generated by the objects $L_t$ for $t \in \mathscr{T}$. We write $\mathscr{C}_{\leq s}$ for $\mathscr{C}_{\{t \in \mathscr{S} \mid t \leq s\}}$, and similarly for $\mathscr{C}_{<s}$. Finally, recall that an \emph{ideal} of $\mathscr{S}$ is a subset $\mathscr{T} \subset \mathscr{S}$ such that if $t \in \mathscr{T}$ and $s \in \mathscr{S}$ are such that $s \leq t$, then $s \in \mathscr{T}$.

Recall that
the category $\mathscr{C}$ (together with the above data) is said to be a \emph{highest weight category} if the following conditions hold:
\begin{enumerate}
\item
\label{it:qh-def-fin}
for any $s \in \mathscr{S}$, the set $\{t \in \mathscr{S} \mid t \leq s\}$ is finite;
\item
\label{it:qh-def-split}
for each $s \in \mathscr{S}$, we have 
$\End_{\mathscr{C}}(L_s) = \bk$;
\item
\label{it:Delta-nabla-proj-inj}
for any $s \in \mathscr{S}$ and any ideal $\mathscr{T} \subset \mathscr{S}$ such that $s \in \mathscr{T}$ is maximal, $\Delta_s \to L_s$ is a projective cover in $\mathscr{C}_{\mathscr{T}}$ and $L_s \to \nabla_s$ is an injective envelope in $\mathscr{C}_{\mathscr{T}}$;
\item
\label{it:qh-def-ker}
the kernel of $\Delta_s \to L_s$ and the cokernel of $L_s \to \nabla_s$ belong to $\mathscr{C}_{<s}$;
\item
\label{it:qh-def-ext2}
we have $\Ext^2_{\mathscr{C}}(\Delta_s, \nabla_t) = 0$ for all $s, t \in \mathscr{S}$.
\end{enumerate}
In this case, the poset $(\mathscr{S}, \leq)$ is called the \emph{weight poset} of $\mathscr{C}$.

See~\cite[\S 7]{riche-hab} for the basic properties of highest weight categories (following Cline--Parshall--Scott and Be{\u\i}linson--Ginzburg--Soergel).

We can finally state the main result of this section.

\begin{thm}
\label{thm:hw-structure}
Assume that $p \nmid |A|$.
The category $\Rep(G)$, equipped with the poset
\[
\left.
\left\{
(\lambda, E) \, \left| \,
\begin{array}{c}
\text{$\lambda \in \bXp$ and $E$ an isom. class} \\
\text{of simple $\sA^\lambda$-modules}
\end{array}
\right\} \right.
\right/ A
\]
(with the order defined in~\eqref{eqn:def-order}) and the objects $\Delta(\lambda,E)$, $L(\lambda,E)$, $\nabla(\lambda,E)$, is a highest weight category.
\end{thm}

\begin{proof}
The desired properties are verified in Theorem~\ref{thm:disconn-class}, Proposition~\ref{prop:standard-disconn} and Lemma~\ref{lem:Ext1-Delta}, their variants for costandard objects (see~\S\ref{ss:costandard}), and Lemma~\ref{lem:Ext-vanishing}.
\end{proof}

%%%%%%%%%%%%%%%%%%%%%%%%%%%%%%%%%%%%%%%%%%%%%%%%%%%%%%%%%%%%%%%%%%%%%%%%%%%
\section{Grothendieck groups}
\label{sec:Groth}
%%%%%%%%%%%%%%%%%%%%%%%%%%%%%%%%%%%%%%%%%%%%%%%%%%%%%%%%%%%%%%%%%%%%%%%%%%%

Our goal in this section is to prove a generalization of a result of Serre~\cite{serre-Groth} providing a description of the Grothendieck group of 
any split connected reductive group over a
strictly Henselian discrete valuation ring of mixed characteristic. (In~\cite{serre-Groth}, the author considers more general coefficients, but we will restrict to a setting which is sufficient for the application we have in mind; see~\cite{ahr}.)

%------------------------------------------------------------------------
\subsection{Setting}
%------------------------------------------------------------------------

We will denote by $\O$ 
a strictly Henselian discrete valuation ring.  We denote its residue field by $\F$, and its fraction field by $\K$.  Recall that $\F$ is separably closed by definition.  We also let $\overline{\F}$ and $\overline{\K}$ be algebraic closures of $\F$ and $\K$, respectively. We will assume that $\K$ has characteristic $0$, and that $\F$ has characteristic $p>0$.

\begin{lem}
\label{lem:reductive-split}
Any reductive group scheme over $\O$ (in the sense of~\cite{sga3.3})
is split.
\end{lem}

\begin{proof}\footnote{This proof, which was communicated to us by Torsten Wedhorn, replaces a more complicated proof that appeared in a previous draft of this paper.}
According to~\cite[Exp.~XXII, Corollaire~2.3]{sga3.3}, any reductive group scheme over $\O$ splits after base change along a suitable \'etale extension $\O \to \O'$.  But because $\O$ is strictly Henselian,~\cite[Proposition~18.8.1(c)]{ega4} tells us that $\O \to \O'$ admits a section. It follows that any reductive group scheme over $\O$ is split.
\end{proof}

In this section we will consider an affine $\O$-group scheme $G$, a closed normal subgroup $G^\circ \subset G$, and we will denote by $A$ the factor group of $G$ by $G^\circ$ in the sense of~\cite[\S I.6.1]{jantzen} (i.e.~of~\cite[III, \S 3, no.~3]{dg}).
We will make the following assumptions:
\begin{enumerate}
\item
\label{it:ass-1}
$G^\circ$ is a reductive group scheme over $\O$ (which is automatically split by Lemma~\ref{lem:reductive-split});
\item
\label{it:ass-2}
$A$ is the constant group scheme associated with a finite group $\mathbf{A}$ (in the sense of~\cite[\S I.8.5(a)]{jantzen}), and moreover $p$ does not divide $|\mathbf{A}|$.
\end{enumerate}

These assumptions have the following consequence.

\begin{lem}
 The $\O$-group scheme $G$ is flat and of finite type.
\end{lem}

\begin{proof}
 By~\cite[III, \S 3, Proposition~2.5]{dg} (see also~\cite[\S I.5.7]{jantzen}), the morphism $G \to A$ is flat and of finite type. Since $A$ is clearly flat and of finite type over $\O$, we deduce the same properties for $G$.
\end{proof}

If $\bk$ is one of $\F$, $\overline{\F}$, $\K$ or $\overline{\K}$, we set
\[
G_\bk := \Spec(\bk) \times_{\Spec(\O)} G, \quad G^\circ_\bk := \Spec(\bk) \times_{\Spec(\O)} G^\circ.
\]
Then by~\cite[Equation~I.5.5(4)]{jantzen}, the quotient $G_\bk / G^\circ_\bk$ is the constant $\bk$-group scheme associated with $\mathbf{A}$; in other words
$G_\bk$ is an extension of the constant (hence smooth) $\bk$-group scheme associated with $\mathbf{A}$ by the smooth group scheme $G^\circ_\bk$. In view of~\cite[Proposition~8.1]{milne} it follows that $G_\bk$ is smooth, and then that $G$ itself is smooth (see~\cite[Tag 01V8]{stacks}).

In particular, the groups $G_{\overline{\F}}$ and $G_{\overline{\K}}$ are algebraic groups (over $\overline{\F}$ and $\overline{\K}$) in the usual ``naive'' sense. Since $G^\circ_{\overline{\F}}$ is connected and $\mathbf{A}$ is finite, the latter group identifies with the group of components of $G_{\overline{\F}}$, and $G^\circ_{\overline{\F}}$ with the identity component of $G_{\overline{\F}}$ (which justifies our notation). Similarly, $\mathbf{A}$ also identifies with the group of components of $G_{\overline{\K}}$, and $G^\circ_{\overline{\K}}$ is the identity component of $G_{\overline{\K}}$.

\begin{lem}
\label{lem:surjectivity}
The morphism $G(\O) \to \mathbf{A}$ induced by $\varpi$ is surjective.
\end{lem}

\begin{proof}
We consider the commutative diagram
\[
\begin{tikzcd}
G(\O) \ar[r] \ar[d] & A(\O) \ar[d, equal] \\
G(\F) \ar[r] & A(\F)
\end{tikzcd}
\]
where the horizontal maps are induced by $\varpi$ and the vertical ones by the quotient morphism $\O \to \F$. Here since $\O$ and $\F$ are integral domains the two groups in the right-hand column identify with $\mathbf{A}$, and the right-hand vertical arrow is an isomorphism. On the other hand, the left-hand vertical arrow is surjective by~\cite[Th\'eor\`eme~18.5.17]{ega4}.

To finish the proof, it remains to show that $G(\F) \to A(\F)$ is surjective. To do this, we consider the diagram
\[
\begin{tikzcd}
G(\F) \ar[r] \ar[d] & A(\F) \ar[d, equal] \\
G(\overline{\F}) \ar[r] & A(\overline{\F}).
\end{tikzcd}
\]
Here the bottom horizontal arrow is surjective, and the right-hand vertical arrow is again an isomorphism.  All the arrows commute with the Frobenius endomorphism, denoted by $\mathrm{Fr}$.  Since $A$ is a constant group scheme, its Frobenius endomorphism is the identity map.  Since $\overline{\F}$ is a purely inseparable extension of $\F$, for any $g \in G(\overline{\F})$, there exists an integer $n \ge 1$ such that $\mathrm{Fr}^n(g)$ lies in the image of $G(\F)$.  The result follows.
\end{proof}

Thanks to Lemma~\ref{lem:surjectivity}, we can (and will) choose a section $\iota : \mathbf{A} \to G(\O)$ of the projection induced by $\varpi$. Of course, we cannot assume that $\iota$ is a group morphism in general; but we will at least assume that $\iota(1)=1$. For simplicity, we will also denote by $\iota : A \to G$ the morphism of $\O$-group schemes defined by $\iota$, i.e.~the $\O$-scheme morphism associated with the algebra morphism $\cO(G) \to \cO(A)=\mathrm{Fun}(\mathbf{A},\O)$ (where $\mathrm{Fun}(\mathbf{A},\O)$ denotes the algebra of functions from $\mathbf{A}$ to $\O$) sending $f$ to the map $a \mapsto \iota(a)(f)$.
Base-changing to $\overline{\F}$ and to $\overline{\K}$, $\iota$ provides sections of the projections of $G_{\overline{\F}}$ and $G_{\overline{\K}}$ onto their respective groups of components.

\begin{lem}
\label{lem:G-product}
The morphism
\[
A \times G^\circ \to G
\]
defined by $(a,g) \mapsto \iota(a) \cdot g$ is an isomorphism of $\O$-schemes.
\end{lem}

\begin{proof}
Consider the algebra morphism $\varphi : \mathcal{O}(G) \to \prod_{a \in \mathbf{A}} \mathcal{O}(G^\circ)$ induced by our morphism; what we have to prove is that $\varphi$ is an isomorphism. From the remarks preceding Lemma~\ref{lem:surjectivity}, we know that the algebra morphisms $\overline{\F} \otimes_\O \varphi$ and $\overline{\K} \otimes_\O \varphi$ are isomorphisms; hence so are the morphisms $\F \otimes_\O \varphi$ and $\K \otimes_\O \varphi$. From the invertibility of $\K \otimes_\O \varphi$ and the fact that $\mathcal{O}(G)$ is flat (hence torsion free) we deduce that $\varphi$ is injective.

Now we denote by $C$ the cokernel of $\varphi$. To prove that $C=0$ it suffices to prove that for any maximal ideal $\mathfrak{m} \subset \mathcal{O}(G)$ the localization $C_{\mathfrak{m}}$ vanishes. Then, since $\prod_{a \in \mathbf{A}} \mathcal{O}(G^\circ)$ is finitely generated as an $\mathcal{O}(G)$-module (because $\mathcal{O}(G^\circ)$ is), so is $C$, so that by Nakayama's lemma it suffices to prove that $C_{\mathfrak{m}} / \mathfrak{m} C_{\mathfrak{m}} = C/\mathfrak{m} C$ vanishes. Now the kernel of the composition $\O \to \mathcal{O}(G)/\mathfrak{m}$ is either $\{0\}$ or the unique maximal ideal $\mathfrak{p}$ in $\O$. In the latter case $C/\mathfrak{m} C$ is a quotient of $C/\mathfrak{p} C=C \otimes_\O \F$, which vanishes since $\F \otimes_\O \varphi$ is an isomorphism. In the former case, the morphism $\O \to \mathcal{O}(G)/\mathfrak{m}$ factors through a morphism $\K \to \mathcal{O}(G)/\mathfrak{m}$, and then $C/\mathfrak{m} C=C \otimes_{\mathcal{O}(G)} \mathcal{O}(G)/\mathfrak{m}$ is a quotient of
\[
C \otimes_\O \mathcal{O}(G)/\mathfrak{m} = (C \otimes_\O \K) \otimes_\K \mathcal{O}(G)/\mathfrak{m},
\]
which vanishes since $\K \otimes_\O \varphi$ is invertible.
\end{proof}

%-----------------------------------------------------------
\subsection{Statement}
%-----------------------------------------------------------

Let us consider the Grothendieck groups
\[
\mathsf{K}(G), \qquad \mathsf{K}(G_\K), \qquad \mathsf{K}(G_\F)
\]
of the categories of (algebraic) $G$-modules of finite type over $\O$, of finite-dimensional (algebraic) $G_\K$-modules, and of finite-dimensional (algebraic) $G_\F$-modules, respectively. We will also denote by $\mathsf{K}_{\mathrm{pr}}(G)$ the Grothendieck group of the exact category of $G$-modules which are free of finite rank over $\O$. Following~\cite{serre-Groth} we consider the commutative diagram of natural morphisms of abelian groups
\begin{equation}
\label{eqn:diagram-Groth-groups}
\begin{tikzcd}
\mathsf{K}_{\mathrm{pr}}(G) \ar[rd] \ar[r, "\sim"] & \mathsf{K}(G) \ar[r, two heads] & \mathsf{K}(G_\K) \ar[ld, "d_G"] \\
& \mathsf{K}(G_\F).
\end{tikzcd}
\end{equation}
Here, on the upper line, the left horizontal map (which is induced by the natural inclusion of categories) is an isomorphism by~\cite[Proposition~4]{serre-Groth}.  The right horizontal map (induced by the exact functor $\K \otimes_\O (-)$) is surjective by~\cite[Th\'eor\`eme~1]{serre-Groth}. The map from the top left-hand corner to the group on the bottom line is induced by the (exact) functor $\F \otimes_\O (-)$. Finally, the map $d_G$ is the ``decomposition'' morphism from~\cite[Th\'eor\`eme~2]{serre-Groth}.

The main result of this section is the following.

\begin{thm}
\label{thm:Groth-gps}
All the maps in~\eqref{eqn:diagram-Groth-groups} are isomorphisms.
\end{thm}

According to~\cite[Th\'eor\`eme~3]{serre-Groth}, if $d_G$ is surjective, then the right-hand morphism on the upper line is automatically an isomorphism. Thus, to prove Theorem~\ref{thm:Groth-gps}, it is enough to prove that $d_G$ is an isomorphism. This will be accomplished in~\S\ref{ss:invertibility} below.

%-------------------------------------------------
\subsection{Lattices}
\label{ss:lattices}
%-------------------------------------------------

Our starting point will be~\cite[Th\'eor\`eme~5]{serre-Groth}, which is applicable here thanks to Lemma~\ref{lem:reductive-split}. This result asserts that if we consider the diagram
\begin{equation}
\label{eqn:diagram-Groth-groups-connected}
\begin{tikzcd}
\mathsf{K}_{\mathrm{pr}}(G^\circ) \ar[rd] \ar[r, "\sim"] & \mathsf{K}(G^\circ) \ar[r, two heads] & \mathsf{K}(G^\circ_\K) \ar[ld, "d_{G^\circ}"] \\
& \mathsf{K}(G^\circ_\F)
\end{tikzcd}
\end{equation}
similar to~\eqref{eqn:diagram-Groth-groups} but for the group $G^\circ$, then the decomposition morphism $d_{G^\circ}$ is an isomorphism, so that all the maps in~\eqref{eqn:diagram-Groth-groups-connected} are isomorphisms.  The main idea of the argument is as follows: first fix a split torus $T \subset G^\circ$ and set $\bX:=X^*(T)$.  Then both $\mathsf{K}(G^\circ_\K)$ and $\mathsf{K}(G^\circ_\F)$ can be embedded in $\Z[\bX]$ by taking characters, and $d_{G^\circ}$ is characterized by the property that it preserves characters.

Let us delve a bit further into the details of the behavior of $d_{G^\circ}$.  Choosing a system of positive roots in the root system of $(G^\circ,T)$, we obtain a Borel subgroup $B \subset G^\circ$ containing $T$ (chosen such that $B$ is the negative Borel subgroup), and a subset $\bX^+ \subset \bX$ of dominant weights. 

By the well-known representation theory of connected reductive groups over algebraically closed fields, both the set of isomorphism classes of simple $G^\circ_{\overline{\K}}$-modules and the set of isomorphism classes of simple $G^\circ_{\overline{\F}}$-modules are in bijection with $\bX^+$. More concretely, if $\lambda \in \bX^+$ and if $L_{\overline{\K}}(\lambda)$ is a simple $G^\circ_{\overline{\K}}$-module of highest weight $\lambda$, then there exists a  simple $G^\circ_{\K}$-module $L_\K(\lambda)$ and an isomorphism $\overline{\K} \otimes_\K L_\K(\lambda) \cong L_{\overline{\K}}(\lambda)$.  Moreover, $L_\K(\lambda)$ is unique up to isomorphism (which justifies the notation), and every simple $G^\circ_\K$-module is of this form. (See e.g.~\cite[\S 3.6]{serre-Groth} or~\cite[Corollary~II.2.9]{jantzen} for details.)

There is a similar description of simple $G^\circ_\F$-modules. Note also that the Weyl and dual Weyl $G^\circ_{\overline{\F}}$-modules have obvious $\F$-versions, that will be denoted $\Delta_\F(\lambda)$ and $\nabla_\F(\lambda)$ respectively.

Next, if $V_\O \subset L_\K(\lambda)$ is a $G^\circ$-stable $\O$-lattice, then the class of $\F \otimes_\O V_\O$ in $\mathsf{K}(G^\circ_\F)$ coincides with the class of the Weyl module $\Delta_{\F}(\lambda)$ of highest weight $\lambda$ (because $L_{\overline{\K}}(\lambda)$ and $\Delta_{\overline{\F}}(\lambda)$ have the same character). In fact, it is well known that the lattice $V_\O$ can be chosen in such a way that $\F \otimes_\O V_\O \cong \Delta_\F(\lambda)$ as $G^\circ_\F$-modules. For each $\lambda$ we will fix such a lattice, and denote it by $L_\O(\lambda)$.  To summarize, we have
\[
d_{G^\circ}([L_\K(\lambda)]) = [\Delta_\F(\lambda)].
\]

In the present setting, $\mathbf{A}$ is the group of components both of $G_{\overline{\F}}$ and of $G_{\overline{\K}}$. Identifying $T_{\overline{\F}}$ and $T_{\overline{\K}}$ with the universal maximal tori of $G^\circ_{\overline{\F}}$ and $G_{\overline{\K}}^\circ$ respectively (via the choice of Borel subgroups obtained from $B$ by base change), we obtain two actions of $\mathbf{A}$ on $\bX=X^*(T_{\overline{\F}})=X^*(T_{\overline{\K}})$, see~\S\ref{ss:disconn}. The description of this action involves the property that Borel subgroups are conjugate, which is not true over $\O$; so it is not clear from the definition that they must coincide. In the next lemma we will show that they do at least coincide on $\bX^+$.

\begin{lem}
\label{lem:actions-coincide}
The two actions of $\mathbf{A}$ on $\bX$ agree on $\bX^+$.
\end{lem}

\begin{proof}
Let us provisionally denote the two actions of $\mathbf{A}$ on $\bX$ by $\cdot_{\overline{\F}}$ and $\cdot_{\overline{\K}}$. Since $\mathbf{A}$ acts by algebraic group automorphisms on $G^\circ$, this group acts on all the Grothen\-dieck groups in~\eqref{eqn:diagram-Groth-groups-connected}, and all the maps in this diagram are obviously $\mathbf{A}$-equivariant. Now for $\lambda \in \bX^+$ we have $a \cdot [L_\K(\lambda)] = [L_\K(a \cdot_{\overline{\K}} \lambda)]$, hence
\[
d_{G^\circ}(a \cdot [L_\K(\lambda)]) = [\Delta_\F(a \cdot_{\overline{\K}} \lambda)].
\]
On the other hand, we have $a \cdot [\Delta_\F(\lambda)] = [\Delta_\F(a \cdot_{\overline{\F}} \lambda)]$. Since $d_{G^\circ}$ is $\mathbf{A}$-equivariant, it follows that
\[
d_{G^\circ}(a \cdot [L_\K(\lambda)]) = a \cdot [\Delta_\F(\lambda)] = [\Delta_\F(a \cdot_{\overline{\F}} \lambda)]
\]
(see~\eqref{eqn:twist-standard}). We deduce that $[\Delta_\F(a \cdot_{\overline{\K}} \lambda)] = [\Delta_\F(a \cdot_{\overline{\F}} \lambda)]$, hence that $a \cdot_{\overline{\K}} \lambda = a \cdot_{\overline{\F}} \lambda$.
\end{proof}

From now on we fix $\lambda \in \bX^+$.
It follows in particular from Lemma~\ref{lem:actions-coincide} that the two possible definitions of the subgroup $\mathbf{A}^\lambda \subset \mathbf{A}$ (see~\S\ref{ss:twisted}) coincide.

\begin{lem}\phantomsection
\label{lem:isom-LO}
\begin{enumerate}
\item
\label{it:End-LO}
We have $\End_{G^\circ}(L_\O(\lambda)) = \O$.
\item
For any $a \in \mathbf{A}^\lambda$, there exists an isomorphism of $G^\circ$-modules
\[
\label{it:Hom-LO}
{}^{\iota(a)} \hspace{-1pt} L_\O(\lambda) \cong L_\O(\lambda).
\]
\end{enumerate}
\end{lem}

\begin{proof}
We only explain the proof of~\eqref{it:Hom-LO}; the proof of~\eqref{it:End-LO} is similar.
Consider the object
\[
R\Hom_{G^\circ}(L_\O(\lambda),{}^{\iota(a)} \hspace{-1pt} L_\O(\lambda))
\]
of the derived category of $\O$-modules.
By~\cite[Lemma~II.B.5 and its proof]{jantzen}, this complex has bounded cohomology, and each of its cohomology objects is finitely generated. This implies that it is isomorphic (in the derived category of $\O$-modules) to a finite direct sum of shifts of finitely generated $\O$-modules.

It follows from~\cite[Proposition~A.6 and Proposition~A.8]{mr} that we have
\begin{align*}
\overline{\F} \lotimes_\O R\Hom_{G^\circ}(L_\O(\lambda),{}^{\iota(a)} \hspace{-1pt} L_\O(\lambda)) &\cong R\Hom_{G^\circ_{\overline{\F}}}(\Delta_{\overline{\F}}(\lambda), \Delta_{\overline{\F}}(\lambda)), \\
\overline{\K} \lotimes_\O R\Hom_{G^\circ}(L_\O(\lambda),{}^{\iota(a)} \hspace{-1pt} L_\O(\lambda)) &\cong R\Hom_{G^\circ_{\overline{\K}}}(L_{\overline{\K}}(\lambda), L_{\overline{\K}}(\lambda)).
\end{align*}
Now we have $R\Hom_{G^\circ_{\overline{\K}}}(L_{\overline{\K}}(\lambda), L_{\overline{\K}}(\lambda)) \cong \overline{\K}$, so that $\Hom_{G^\circ}(L_\O(\lambda), {}^{\iota(a)} \hspace{-1pt} L_\O(\lambda))$ is a sum of $\O$ and a torsion module. But since $\Hom_{G^\circ_{\overline{\F}}}(\Delta_{\overline{\F}}(\lambda),\Delta_{\overline{\F}}(\lambda))={\overline{\F}}$, this torsion module is zero; in other words we have $\Hom_{G^\circ}(L_\O(\lambda), {}^{\iota(a)} \hspace{-1pt}L_\O(\lambda)) \cong \O$. If $f : L_\O(\lambda) \to {}^{\iota(a)} \hspace{-1pt} L_\O(\lambda)$ is a generator of this rank-$1$ $\O$-module, the $G_{{\overline{\F}}}$-module morphism ${\overline{\F}} \otimes_\O f$ is an isomorphism, so that $f$ is also an isomorphism.
\end{proof}

%-------------------------------------------------
\subsection{Comparison of twisted group algebras}
\label{ss:comparison}
%-------------------------------------------------

We continue with the setting of~\S\ref{ss:lattices} (and in particular with our fixed $\lambda \in \bX^+$).

By Lemma~\ref{lem:isom-LO} we can choose, for any $a \in \mathbf{A}^\lambda$, an isomorphism $\theta_a : L_\O(\lambda) \simto {}^{\iota(a)} \hspace{-1pt} L_\O(\lambda)$. Then $\overline{\K} \otimes_\O \theta_a$ is an isomorphism from $L_{\overline{\K}}(\lambda)$ to ${}^{\iota(a)} \hspace{-1pt} L_{\overline{\K}}(\lambda)$, and for $a,b \in \mathbf{A}^\lambda$ the scalar $\alpha(a,b) \in \overline{\K}$ defined in~\S\ref{ss:twisted} using these isomorphisms in fact belongs to $\O^\times$. In particular, if $\mathscr{A}^\lambda_{\overline{\K}}$ is the associated twisted group algebra (over $\overline{\K}$), then the $\O$-lattice $\mathscr{A}^\lambda_\O := \bigoplus_{a \in A^\lambda} \O \cdot \rho_a$ is an $\O$-subalgebra in $\mathscr{A}^\lambda_{\overline{\K}}$. On the other hand, ${\overline{\F}} \otimes_\O \theta_a$ is an isomorphism from $\Delta_{\overline{\F}}(\lambda)$ to ${}^{\iota(a)} \hspace{-1pt} \Delta_{\overline{\F}}(\lambda)$, and by Remark~\ref{rmk:Alambda-standards} the algebra $\mathscr{A}^\lambda_{\overline{\F}}$ from~\S\ref{ss:twisted} (now for the group $G_{\overline{\F}}$ and its simple module $L_{\overline{\F}}(\lambda)$) can be described as the twisted group algebra of $\mathbf{A}^\lambda$ defined by the cocyle sending $(a,b)$ to the image of $\alpha(a,b)$ in ${\overline{\F}}$.

Summarizing, we have obtained an $\O$-algebra $\mathscr{A}^\lambda_\O$ which is free over $\O$ and such that
\[
\overline{\K} \otimes_\O \mathscr{A}^\lambda_\O \cong \mathscr{A}^\lambda_{\overline{\K}}, \quad {\overline{\F}} \otimes_\O \mathscr{A}^\lambda_\O \cong \mathscr{A}^\lambda_{\overline{\F}}.
\]
From Lemma~\ref{lem:semisimple-algebra} we know that $\mathscr{A}^\lambda_{\overline{\F}}$ and $\mathscr{A}^\lambda_{\overline{\K}}$ are products of matrix algebras (over ${\overline{\F}}$ and $\overline{\K}$ respectively). In fact, the same arguments show that $\mathscr{A}^\lambda_\F:=\F \otimes_\O \mathscr{A}^\lambda_\O$ and $\mathscr{A}^\lambda_\K:=\K \otimes_\O \mathscr{A}^\lambda_\O$ are also products of matrix algebras (over $\F$ and $\K$ respectively). Hence we are in the setting of Tits' deformation theorem (see e.g.~\cite[Theorem~7.4.6]{geck-pfeiffer}), and we deduce that we have
a canonical bijection between the sets of isomorphism classes of simple $\mathscr{A}^\lambda_\K$-modules and isomorphism classes of simple $\mathscr{A}^\lambda_\F$-modules, which sends a simple module $M$ to $\F \otimes_\O M_\O$, where $M_\O$ is any $\mathscr{A}^\lambda_\O$-stable $\O$-lattice in $M$.

If $E$ be a finite-dimensional $\mathscr{A}^\lambda_\K$-module, the same procedure as in~\S\ref{ss:lambda-E} allows us to define a $G^\lambda_\K$-module structure on $E \otimes_\K L_\K(\lambda)$, where $G^\lambda_\K$ is the inverse image of $\mathbf{A}^\lambda$ under the map $G_\K \to A$ induced by $\varpi$. Similarly, copying the definitions in Lemma~\ref{lem:Delta-disconn} and Proposition~\ref{prop:standard-disconn}, if $E'$ be a finite-dimensional $\mathscr{A}^\lambda_\F$-module, then we can consider the $G_\F$-module $\Delta_\F(\lambda,E')$, which is an $\F$-form of $\Delta_{\overline{\F}}(\lambda,\overline{\F} \otimes_\F E')$.

\begin{lem}
\label{lem:d-simples}
Let $E$ be a simple $\mathscr{A}^\lambda_\K$-module, and let $\tilde{E}$ be the simple $\mathscr{A}^\lambda_\F$-module corresponding to $E$ under the bijection above. Then we have
\[
d_G([\Ind_{G^\lambda_\K}^{G_\K}(E \otimes_\K L_\K(\lambda))]) = [\Delta_\F(\lambda,\tilde{E})].
\]
%where $G^\lambda_\F$ is the inverse image of $\mathbf{A}^\lambda$ under the map $G_\F \to A$ induced by $\varpi$.
\end{lem}

\begin{proof}
If $E_\O \subset E$ is an $\mathscr{A}^\lambda_\O$-stable $\O$-lattice in $E$, then $E_\O \otimes_\O L_\O(\lambda)$ has a natural structure of $G^\lambda$-module (where $G^\lambda=\varpi^{-1}(\mathbf{A}^\lambda)$), and is a $G^\lambda$-stable $\O$-lattice in $E \otimes_\K L_\K(\lambda)$. Inducing to $G$, we deduce that $\Ind_{G^\lambda}^{G}(E_\O \otimes_\O L_\O(\lambda))$ is a $G$-stable $\O$-lattice in $\Ind_{G^\lambda_\K}^{G_\K}(E \otimes_\K L_\K(\lambda))$, whose modular reduction is $\Delta_\F(\lambda,\tilde{E})$.
\end{proof}

%-------------------------------------------------
\subsection{Invertibility of \texorpdfstring{$d_G$}{dG}}
\label{ss:invertibility}
%-------------------------------------------------

We can now prove that $d_G$ is an isomorphism, which will finish the proof of Theorem~\ref{thm:Groth-gps}.

In fact, for any $\lambda \in \bX^+$, since $\mathscr{A}^\lambda_\K$ is a product of matrix algebras the assignment $E \mapsto \overline{\K} \otimes_\K E$ induces a bijection between the sets of isomorphism classes of simple modules for the algebras $\mathscr{A}^\lambda_\K$ and $\mathscr{A}^\lambda_{\overline{\K}}$ from~\S\ref{ss:comparison}. Then, using Theorem~\ref{thm:disconn-class} and arguing as in~\cite[\S 3.6]{serre-Groth}, we see that the similar operation induces a bijection between the sets of isomorphism classes of simple $G_\K$-modules and of simple $G_{\overline{\K}}$-modules.

The same construction gives a bijection between the sets of isomorphism classes of simple $G_\F$-modules and of simple $G_{\overline{\F}}$-modules.

Let us now fix a subset $\Lambda \subset \bX^+$ of representatives for the $\mathbf{A}$-orbits on $\bX^+$.
By the remarks above, the classes of the modules $\Ind_{G^\lambda_\K}^{G_\K}(E \otimes L_\K(\lambda))$, where $(\lambda,E)$ runs over the pairs consisting of an element $\lambda \in \Lambda$ and a simple $\mathscr{A}^\lambda_\K$-module, form a basis of $\mathsf{K}(G_\K)$ (see in particular Remark~\ref{rmk:classification-simples}\eqref{it:rmks-representatives}). In view of Lemma~\ref{lem:d-simples}, Theorem~\ref{thm:hw-structure}, and the preceding paragraph, the image of this basis under $d_G$ is a basis of $\mathsf{K}(G_\F)$. Hence, $d_G$ is indeed an isomorphism.

%%%%%%%%%%%%%%%%%%%%%%%%%%%%%%%%%%%
%%%%%%%%%%%%%%%%%%%%%%%%%%%%%%%%%%%

\end{document}